\documentclass[11pt,reqno]{amsart}
\usepackage{tikz}
\usepackage{subfig}
\usepackage{epstopdf}
\usepackage{epsfig}
\usepackage{hyperref}
\usepackage{math}
\numberwithin{equation}{section}

\graphicspath{{./figures/}}
\usepackage{tikz}
\usetikzlibrary{shapes,arrows}
\tikzstyle{decision} = [diamond, draw, fill=blue!20, 
    text width=4.5em, text badly centered, node distance=3cm, inner sep=0pt]
\tikzstyle{block} = [rectangle, draw, fill=blue!20, 
    text width=5em, text centered, rounded corners, minimum height=4em]
\tikzstyle{line} = [draw, -latex']
\tikzstyle{cloud} = [draw, ellipse,fill=red!20, node distance=3cm,
    minimum height=2em]
\usetikzlibrary{positioning}
\tikzset{main node/.style={circle,fill=blue!20,draw,minimum size=1cm,inner sep=0pt},  }

\newcommand{\la}{\langle}
\newcommand{\ra}{\rangle}

\newcommand{\ts}{\mathsf{T}}

\usepackage{stackengine} 
\newcommand\oast{\stackMath\mathbin{\stackinset{c}{0ex}{c}{0ex}{\ast}{\bigcirc}}}

\newcommand{\ba}{\begin{array}}
\newcommand{\ea}{\end{array}}
\newcommand{\be}{\begin{equation}}
\newcommand{\ee}{\end{equation}}
\newcommand{\bea}{\begin{eqnarray}}
\newcommand{\eea}{\end{eqnarray}}
\newcommand{\beaa}{\begin{eqnarray*}}
\newcommand{\eeaa}{\end{eqnarray*}}

\def\cE{{\mathcal E}}

\def\hR{\mathbb{R}}

\def\q{\quad}
\def\qq{\qquad}

\def\pa{\partial}

\def\qed{ \hfill \vrule width.25cm height.25cm depth0cm\smallskip}

\newcommand{\basa}{\begin{assumption}}
\newcommand{\easa}{\end{assumption}}

\newcommand{\bas}{\begin{assum}}
\newcommand{\eas}{\end{assum}}

\def\1{{\bf 1}}

\def\:{\!:\!}

at 9pt

\begin{document}
\title[]{Exponential Entropy dissipation for weakly self-consistent Vlasov-Fokker-Planck equations}
\author[Bayraktar]{Erhan Bayraktar$^\sharp$}
\author[Feng]{Qi Feng$^*$}
\author[Li]{Wuchen Li$^\dagger$}
\thanks{$^\sharp$Department of
Mathematics, University of Michigan, Ann Arbor, MI, 48109; Email: erhan@umich.edu.}
\thanks{$^*$Department of
Mathematics, Florida State University, Tallahassee, FL, 32306; Email: qfeng2@fsu.edu,  qif@umich.edu.}
\thanks{$^\dagger$Department of Mathematics, University of South Carolina, Columbia, SC 29208; Email: wuchen@mailbox.sc.edu.}
\thanks{E. Bayraktar is partially supported by the National Science Foundation under grant DMS-2106556 and by the Susan M. Smith chair. Q. Feng is partially supported by the National Science Foundation under grant DMS-2306769. W. Li is supported by AFOSR MURI FA9550-18-1-0502,  AFOSR YIP award No. FA9550-23-1-0087, NSF RTG: 2038080, and NSF DMS-2245097.}

\begin{abstract}
We study long-time dynamical behaviors of weakly self-consistent Vlasov-Fokker-Planck equations. We introduce Hessian matrix conditions on mean-field kernel functions, which characterizes the exponential convergence of solutions in $L^1$ distances. 
The matrix condition is derived from the dissipation of a selected Lyapunov functional, namely auxiliary Fisher information functional. We verify proposed matrix conditions in examples. 
\end{abstract}
\maketitle
\section{Introduction}
Weakly self-consistent Vlasov-Poisson-Fokker-Planck equations \cite{bolley2010trend, guillin2021kinetic,villani2006} play essential roles in mathematical physics and probability with applications in modeling and machine learning sampling problems. The equation describes the probability density's evolution of particles, which interact with each other from interaction energies while under white noise perturbations. 

Consider a mean--field underdamped Langevin diffusion process 
\bea\label{variable underdamped langevin}
\begin{cases}
dx_t&= v_tdt\nonumber \\
dv_t&= -v_t dt-(\int_{\T^{d}\times\hR^d}\nabla_x W(x_t,y)f(t,y,\widetilde v)d\widetilde vdy+\nabla_xU(x_t) )dt+\sqrt{2} dB_t,
\end{cases}
\eea
where $(x_t,v_t)\in \mathbb{T}^d\times \mathbb{R}^d$ presents an identical particle's position and velocity, $\mathbb{T}^d$ is a $d$ dimensional torus representing a position domain, and $B_t$ is a standard Brownian motion in $\hR^d$. Each identical particle interacts with each other through a mean field interaction potential $W$ and a confinement potential $U$. We assume that $W\in  C^{\infty}(\T^{d}\times \T^d)$ is a symmetric kernel function, that is, $W(x,y)=W(y,x)$. Denote $f=f(t,x,v)$ as the probability density function of the stochastic process $(x_t,v_t)$. The density function $f$ follows a non-linear Fokker-Planck equation:
\bea\label{Kinetic FP}
\pa_tf+v\cdot \nabla_xf-(\int_{\T^{d}\times\hR^d}\nabla_x W(x,y)f(t,y,\widetilde v)d\widetilde vdy+\nabla_xU(x))\cdot\nabla_vf=\nabla_v\cdot( f v) +\nabla_v\cdot(\nabla_vf).\nonumber\\
\eea
An equilibrium of equation \eqref{Kinetic FP} satisfies  
\begin{equation*}
    f_{\infty}(x,v)=\frac{1}{Z}e^{-\frac{\|v\|^2}{2}-\int_{\T^{d}\times\hR^d} W(x,y)f_{\infty}(y,\tilde v)dyd\tilde v -U(x)},
\end{equation*}
where $Z=\int_{\T^d\times\hR^d} e^{-\frac{\|v\|^2}{2}-\int_{\T^{d}\times\hR^d} W(x,y)f_{\infty}(y,\tilde v)dyd\tilde v-U(x)} dxdv<+\infty$ is a normalization constant. 
This is known as a nonlinear Gibbs distribution. We are interested in studying long-time behaviors of density functions. How fast does function $f$ converge to $f_\infty$? 

In this paper, we establish an exponential convergence result for the solution of \eqref{Kinetic FP} in {both functional  free energy and} $L^1$ distances. Our method follows from a Lyapunov method, where the Lyapunov functional is selected as auxiliary Fisher information functionals. From the dissipation of Lyapunov functionals along with PDE \eqref{Kinetic FP}, we derive a matrix condition, which guarantees the exponential convergence decay result in {both functional free energy and}  $L^1$ distances. Explicit examples are studied. 

 In the literature, various properties of the Vlasov-Fokker-Planck equation have been studied, e.g.   \cite{Bouchut, BouchutDolbeault, carrilloSoler, Degond, esposito2010stability}. The original Vlasov equation involves the inverse of the Laplacian operator as the interaction kernel function. For simplicity, we only focus on the weakly self-consistent kernel functions, where $W$ is a given smooth function.  The existence of a smooth solution and the regularity property have recently been studied in \cite{cesbron2023}, see also the general approach regarding the regularity properties in \cite{villani2006}[Appendix A.21].   The convergence for particle systems with mean field interactions has been studied in \cite{guillin2019uniform, guillin2021kinetic}. And \cite{guillin2021kinetic} proves the exponential convergence in $H^1$ norms;  \cite{bolley2010trend, wang2021exponential} show exponential convergence results in Wasserstein-$2$ type distances. In addition, \cite{guillin2021convergence} studies the exponential contraction of the solution in the Wasserstein-$1$ distance for nonconvex confinement potential energies. 
{ We remark that there are comparison studies and discussions between $H^1$ and $L^1$ for the Fokker-Planck equations; see the detailed argument in \cite{MC2000}. One important fact of the $L^1$ norm is that this ensures the density has finite mass in physics.
 In particular, the $L^1$ distance is closely related to the Helmholtz free energy for physical systems through the Csisz\`ar-Kullback inequality or the Pinsker inequality. Our convergence analysis of the functional  free energy is crucial for statistical physics-oriented equations, e.g.: spatially homogeneous Fokker-Planck-Landau equation in plasma physics \cite{desvillettes2000}, Fokker-Planck equation for granular media \cite{benedetto1998non}, etc.}
Furthermore, our methods are closely related to but technically different from Villani's hypocoercivity methods \cite{villani2006}. Villani's methods estimate the first-order dissipation, which prove the $O(t^{-\infty})$ decay for $W(x,y)=W(x-y)$; see \cite[Theorem 56]{villani2006}. Meanwhile, we estimate the second-order dissipation, and obtain a Hessian matrix condition of interaction kernel function $W$ and potential function $U$ to determine the exponential convergence result. As in \cite{FengLi2023, FengLi2021}, we develop entropy dissipation methods \cite{arnold2000generalized, arnold2014sharp,arnold2000large, arnold2001convex, bakry1985,baudoin2016curvature, calogero2012exponential, carrillo2003kinetic, CLZ,Li2021_transportb} for equation \eqref{Kinetic FP}.     
It is worth mentioning that the convergence analysis of mean-field Langevin dynamics and nonlinear Fokker-Planck equations are important in AI (artificial intelligence) sampling problems \cite{carrillo2021consensus, garbuno2020interacting, hu2019mean,kazeykina2020ergodicity, li2021controlling}.  This is to design mean-field Markov--Chain--Monte--Carlo (MCMC) sampling algorithms. The convergence analysis of $f$ towards $f_\infty$ plays a key role in AI theory. It helps in designing algorithmic reliable kinetic sampling methods \cite{FengLi2021,li2021controlling,ma2019there} in Bayesian inverse problems. In this direction, our result estimates the exponential decay rates in {both functional free energy and} $L^1$ distance for kinetic type degenerate mean-field stochastic processes. { Our explicit condition for the potential function $V$ and the interaction kernel $W$ are more straightforward to verify in the applications than those using the Logarithmic Sobolev type inequalities as the assumption.}

This paper is organized as follows. In Section \ref{section2}, we present the main result of this paper. In particular, we give the Hessian matrix conditions for \eqref{Kinetic FP}, under which we establish the global exponential convergence results in time in both the functional free energy and the norm $L^1$. In section \ref{section3}, we verify the proposed conditions in examples.  
In sections \ref{section4} and \ref{section5}, we provide the proofs of the main results.

\section{Main results}\label{section2}
In this section, we present the main result of this paper. We first introduce the notions of free energy, Fisher information, and an auxiliary functional, which are the Lyapunov functionals in this paper. We next present the main theorem, which holds under a mean field information matrix condition. This is based on a Lyapunov method, under which we derive an exponential Lyapunov constant for PDE \eqref{Kinetic FP}. We last present examples. 
\subsection{Notations}
We briefly introduce the analytical property for the solution of PDE \eqref{Kinetic FP}. 
Assume that $W(x,y)=W(y,x)$ and $W\in C^{2}(\mathbb{T}^d\times\mathbb{T}^d)$. Denote $f_0=f_0(x,v)$ as a probability density on $\Omega:=\mathbb{T}^d\times\mathbb{R}^d$, such that all moments of density function $f_0$ are finite, that is, $\int_\Omega \|v\|^kf_0(x,v)dxdv<+\infty$, for all $k\in \mathbb Z_+$. We assume that there exists a unique smooth solution $f = f(t,x,v)$ of equation \eqref{Kinetic FP}. See details in \cite{villani2006}. 

We consider a Lyapunov functional for \eqref{Kinetic FP}, which is often named free energy. For convenience of presentation, denote a probability density space supported on $\Omega=\mathbb{T}^d\times\mathbb{R}^d$.
\begin{equation*}
    \mathcal{P}=\Big\{f\in L^1(\Omega)\colon \int_{\Omega}f(x,v)dxdv=1,\quad f\geq 0\Big\}. 
\end{equation*}
\begin{definition}[Free energy]
Define a functional $\mathcal{E}\colon \mathcal{P}\rightarrow\mathbb{R}$ as
\bea\label{free energy}
\mathcal{E}(f)&=&\int_{\Omega} f(x,v)\log f(x,v) dxdv+ \int_{\Omega}  \frac{1}{2}\|v\|^2f(x,v)dxdv\nonumber \\
&&+\frac{1}{2}\int_{\Omega\times\Omega} W(x,y)f(x,v)f(y,\widetilde v)dx dvdyd\widetilde v+\int_{\Omega}U(x)f(x,v)dxdv.
\eea
\end{definition}
In this paper, we study the convergence behavior of \eqref{Kinetic FP} through functional $\mathcal{E}(f)$. We check that when $f_{\infty}$ is the minimizer of $\mathcal{E}$, then 
\begin{equation*}
\frac{\delta}{\delta f(x,v)}\mathcal{E}(f)|_{f=f_{\infty}}=\log f_\infty(x,v)+1+\frac{1}{2}\|v\|^2+\int_\Omega W(x,y)f_\infty(y,\tilde v)dyd\tilde v+U(x)=C, 
\end{equation*}
where $\frac{\delta}{\delta f}$ is the $L^2$ first variation operator w.r.t. $f$, and $C$ is a constant. In other words,  
\begin{equation*}
f_{\infty}(x,v)=\frac{1}{Z}e^{-\frac{\|v\|^2}{2}-\int_{\Omega} W(x,y)f_{\infty}(y,\tilde v)dyd\tilde v-U(x)},
\end{equation*}
where $Z=\int_{\Omega} e^{-\frac{\|v\|^2}{2}-\int_{\Omega} W(x,y)f_{\infty}(y,\tilde v)dyd\tilde v-U(x)} dxdv<+\infty$ is a normalization constant. 
In literature, we note that functional $\mathcal{E}(f)$ is often named the {\em free energy}, and $f_\infty$ is called a {\em nonlinear Gibbs distribution} \cite{carlen2003free, esposito2010stability}. 

In this paper, we mainly study the long time dynamical behavior of $f(t,x,v)$ for general interaction kernel function $W$ and potential function $U$. In particular, we shall investigate that how fast does the Lyapunov functional $\mathcal{E}(f(t,\cdot))$ converge to $\mathcal{E}(f_{\infty})$.  

To study this convergence, we introduce some necessary notations and functionals. Denote $\mathsf I_d\in\hR^{d\times d}$ as the identity matrix.
Let
\bea
\label{matrix a z}
a=\begin{pmatrix}
0
\\
\mathsf I_d
\end{pmatrix}\in \hR^{2d\times d},\quad z=\begin{pmatrix}z_1 \mathsf I_d\\ z_2\mathsf I_d\end{pmatrix}\in \hR^{2d\times d}, 
\eea
where $z_1,z_2\in\hR$ are two given constants. 
Using above matrices, we define the following functionals to characterize the decay of Lyapunov functional $\mathcal{E}$. 
Denote 
\begin{equation*}
    \frac{\delta}{\delta f(x,v)}\mathcal{E}(f)=\log f(x,v)+1+\frac{1}{2}\|v\|^2+\int_\Omega W(x,y)f(y,\tilde v)dyd\tilde v+U(x). 
\end{equation*}
\begin{definition}[Fisher information functionals]
Define a functional $\mathcal{DE}_a\colon \mathcal{P}\rightarrow\mathbb{R}_+$ as
\bea\label{fisher}
\mathcal{DE}_a(f)&:=&\int_{\Omega} \la\nabla_{x,v} \frac{\delta}{\delta f(x,v)}\mathcal{E}(f), aa^{\ts}\nabla_{x,v} \frac{\delta}{\delta f(x,v)}\mathcal{E}(f)\ra f(x,v) dxdv.
\eea 
Define an auxiliary functional $\mathcal{DE}_z\colon \mathcal{P}\rightarrow\mathbb{R}_+$ as 
\bea\label{fisher z}
\mathcal{DE}_z(f)&:=&\int_{\Omega} \la\nabla_{x,v} \frac{\delta}{\delta f(x,v)}\mathcal{E}(f), zz^{\ts}\nabla_{x,v} \frac{\delta}{\delta f(x,v)}\mathcal{E}(f)\ra f(x,v) dxdv.
\eea
\end{definition}
It is known that $\mathcal{DE}_a$, named ``Fisher information functional'', equals to the decay of free energy $\mathcal{E}$ along with the solution of PDE \eqref{Kinetic FP}. In other words, 
\begin{equation}\label{EV}
\frac{d}{dt}\mathcal{E}(f(t,\cdot,\cdot))=-\mathcal{DE}_a(f(t,\cdot,\cdot))\leq 0.     
\end{equation}
 This result is stated in Lemma \ref{lemma energy decay}. We note that functional $\mathcal{DE}_a$ itself can not guarantee the $L^1$ decay of the solution, due to the degeneracy of the subelliptic operator in PDE \eqref{Kinetic FP}. To overcome this degeneracy issue, we construct an additional functional $\mathcal{DE}_z$. We call $\mathcal{DE}_z$ the ``auxiliary Fisher information functional''.
Shortly, we demonstrate that the designed auxiliary functional $\mathcal{DE}_z$ is useful in establishing the decay rate of the degenerate subelliptic operator in \eqref{Kinetic FP}. 
\subsection{Main result}
We are ready to present the main result. We first provide a matrix eigenvalue assumption. 
\begin{definition}[Mean--field information matrix]\label{definition mean field info matrix}
Define a symmetric matrix function $\mathfrak R\in\mathbb{R}^{4d\times 4d}$, such that
\beaa 
\mathfrak R(z,x,y)=\frac{1}{2}\begin{pmatrix}
\mathsf{A}(x,y)&\mathsf{B}(x,y)\\
\mathsf{B}(x,y)&\mathsf{A}(y,x)
\end{pmatrix},
\eeaa 
where $\mathsf{A}$, $\mathsf{B}\in\mathbb{R}^{2d\times 2d}$ 
are defined below: Denote 
$$V(x,y)=U(x)+W(x,y),$$
\beaa
\mathsf A(x,y)&=&\begin{pmatrix}
	z_1z_2\mathsf I_d& \frac{1}{2}[(1+z_1z_2+z_2^2)\mathsf I_d-z_1^2\nabla^2_{xx}V(x,y)]\\
	\frac{1}{2}[(1+z_1z_2+z_2^2)\mathsf I_d-z_1^2\nabla^2_{xx}V(x,y)]& (1+z_2^2)\mathsf I_d-z_1z_2\nabla^2_{xx} V(x,y)
	\end{pmatrix},\\
	\mathsf A(y,x)&=&\begin{pmatrix}
	z_1z_2\mathsf I_d& \frac{1}{2}[(1+z_1z_2+z_2^2)\mathsf I_d-z_1^2\nabla^2_{yy}V(y,x)]\\
	\frac{1}{2}[(1+z_1z_2+z_2^2)\mathsf I_d-z_1^2\nabla^2_{yy}V(y,x)]& (1+z_2^2)\mathsf I_d-z_1z_2\nabla^2_{yy} V(y,x)
	\end{pmatrix},
	\eeaa
 and	
\beaa	
	\mathsf B(x,y)=\begin{pmatrix}
0&-\frac{z_1^2}{2}\nabla^2_{xy}W(x,y)\\
	-\frac{z_1^2}{2}\nabla^2_{xy}W(x,y)&-z_1z_2\nabla^2_{xy}W(x,y)
\end{pmatrix}.
\eeaa
\end{definition}

\begin{assum}[Mean--field information matrix condition]\label{main assumption} 
Assume that there exists a constant $\lambda>0$, such that
\bea\label{lambda}
\mathfrak R(z,x,y)\succeq \lambda \begin{pmatrix}
    aa^{\ts}+zz^{\ts}&0\\
    0&aa^{\ts}+zz^{\ts}
\end{pmatrix},
\eea
where $a$, $z$ are defined in \eqref{matrix a z}, such that
\beaa 
aa^{\ts}+zz^{\ts}=\begin{pmatrix}
	z_1^2\mathsf I_d&z_1z_2\mathsf I_d\\
	z_1z_2 \mathsf I_d&(1+z_2^2)\mathsf I_d
\end{pmatrix}.
\eeaa 
\end{assum}
Under the above matrix eigenvalue condition, we next prove the following main theorem. 
\begin{theorem}\label{main theorem} Suppose that Assumption \ref{main assumption} holds, and there exists a smooth solution $f(t,x,v)$ of \eqref{Kinetic FP}. Then the following exponential convergence result is satisfied. 
\bea\label{main_convergence}
\mathcal{E}(f(t,\cdot,\cdot))-\mathcal E(f_{\infty})\le \frac{1}{2\lambda}e^{-2\lambda t}[\mathcal{DE}_{a,z}(f_0)-\mathcal{DE}_{a,z}(f_{\infty}) ],
\eea	
where 
\bea \label{main_convergence_1}
\mathcal{DE}_{a,z}(f):=\mathcal{DE}_a(f)+\mathcal{DE}_z(f).
\eea 
\end{theorem}
\begin{corollary}\label{cor 2}
Suppose Assumption \ref{main assumption} holds, and there exists a smooth solution $f(t,x,v)$ of \eqref{Kinetic FP}. Assume that there exists a sufficient small constant $C_{W}>0$, such that 
\bea\label{assump bound on W}
\max_{(x,y)\in\T^d\times\T^d}|W(x,y)|\le C_W.
\eea 
Then the following $L^1$ distance convergence holds. 
	\begin{equation*}
\int_\Omega|f(t,x,v)-f_{\infty}(x,v)|dxdv\le Ce^{-\lambda t}\sqrt{\mathcal{DE}_{a,z}(f_0)-\mathcal{DE}_{a,z}(f_{\infty})},
\end{equation*} 
for some constant $C>0$.
\end{corollary} 
\begin{remark}
The functional $\mathcal{E}$ is also significant in physics. It is the Helmholtz free energy. Thus, the convergence analysis of the kinetic Fokker-Planck equation in terms of Helmholtz free energy is crucial for statistical physics-oriented equations.
The second law of thermodynamics shows that free energy dissipation equals the negative Fisher information; see Lemma \ref{lemma energy decay}. The analysis in this paper further introduces the convergence rate of the free energy. Regarding Lyapunov's functions, there are other choices, including $H^{-1}$ distances. To our knowledge, they are not oriented for analyzing the Helmholtz-free energy. We refer interesting readers to \cite{MC2000} for the importance of $L^1$ distances and free energy estimations.
\end{remark}

\subsection{Examples}
We last present two concrete examples of $L^1$ distance exponential convergence results for different kernels $W$ and potentials $U$ in \eqref{Kinetic FP}. We leave their proofs with detailed conditions in section \ref{section3}.  

\begin{example}\label{ex3}
Assume $W(x,y)=W(y,x)$, $U(x)\neq 0$. For $z_1=1$ and $z_2=0.3$, assume 
\beaa 
\underline{\lambda}\mathsf I_{d}\preceq\nabla^2_{xx}(U(x)+W(x,y))\preceq\overline{\lambda}\mathsf I_d, \quad 2\underline{\lambda}-\overline{\lambda}^2>0.08,
\eeaa 
and for sufficiently small $\varepsilon>0$, 
\begin{equation*}
{ \|\nabla^2_{x,y} W(x,y)\|_{\mathrm{F}}=O(\varepsilon)},\quad \max_{(x,y)\in\T^d\times\T^d}|W(x,y)|<1,
\end{equation*}
where $\|\cdot\|_{\mathrm{F}}$ is the matrix Frobenius norm. Then Assumption \ref{main assumption} is satisfied. Thus the exponential convergence results in \eqref{main_convergence} and \eqref{main_convergence_1} hold.
\end{example}

\begin{example}\label{ex2}
Consider $W(x,y)=W(y,x)$, $U(x)\neq 0$. Assume 
\beaa
\begin{cases}
\nabla^2_{xy}W(x,y)=\mathsf Q^{-1}_W\textbf{Diag}\Big(\lambda_1^W,\cdots,\lambda_d^W \Big)\mathsf Q_W,\\
\nabla^2_{xx}W(x,y)= \mathsf Q^{-1}_W\textbf{Diag}\Big(\widetilde\lambda_{1}^{W},\cdots,\widetilde\lambda_{d}^W \Big)\mathsf Q_W,\\
\underline\lambda\mathsf{I}_d 	\preceq \nabla^2_{xx}U(x) 	\preceq     \overline\lambda\mathsf{I}_d,
\end{cases}
\eeaa 
where $\mathsf Q_W$ denotes the orthogonal matrix for the eigenvalue decomposition of $\nabla^2W$. Let $\underline\lambda=\overline{\lambda}=0.9$, $z_1=1$, and $z_2=0.3$. If the following condition holds,
\beaa 
-0.538< \widetilde\lambda_i^W<  0.297,\quad i=1,\cdots, d, \quad \max_{(x,y)\in\T^d\times\T^d}|W(x,y)|<1,
\eeaa 
then for sufficiently small $\lambda_i^W$, Assumption \ref{main assumption} holds true. Thus the exponential convergence results in \eqref{main_convergence} and \eqref{main_convergence_1} hold. In particular, for $d=1$, if $\widetilde\lambda_i^W=-0.12$, then $|\lambda_i^W|<10^{-3}$ is enough to guarantee Assumption \ref{main assumption}. If $W(x,y)=W(x-y)$ and $\widetilde\lambda_i^W=-\lambda_i^W$ is small enough, Assumption \ref{main assumption} holds.
\end{example}

\begin{remark}[Comparisons with \cite{carrillo2003kinetic}]
The mentioned paper studies a matrix eigenvalue condition for gradient-drift Fokker-Planck equation. In example \ref{ex2}, we work on a matrix eigenvalue condition for degenerate non-gradient drift Fokker-Planck equation. 
\end{remark}

\begin{remark}[Comparisons with \cite{guillin2021kinetic}] The paper in this remark establishes the exponential convergence results in weighted Sobolev space. Meanwhile, we show the exponential convergence results in $L^1$ distance. 
\end{remark}

\begin{remark}\label{rem 3}
The work of \cite{villani2006} analyzes the case for $U\equiv 0$ and $W(x,y)=W(x-y)$. We shall show that assumption \ref{main assumption} does not hold for constant matrices $a$ and $z$. This implies that exponential decay does not hold in this example. In this sense, our result does not improve the $O(t^{-\infty})$ convergence result in \cite{villani2006}[Theorem 56]. 
\end{remark}

\section{Verification of assumptions in examples}\label{section3}
In this section, we verify Assumption \ref{main assumption} in two examples.

\subsection{Proof of Example \ref{ex3}}
\begin{lemma}\label{lemma: W x - y}
    Assume that
\bea\label{assump: hesss W xy-1}
\begin{cases}
\nabla^2_{xy}W(x,y)=\mathsf Q^{-1}_W\textbf{Diag}\Big(\lambda_1^W,\cdots,\lambda_d^W \Big)\mathsf Q_W,\\
\nabla^2_{xx}W(x,y)+\nabla^2_{xx}U(x)=\mathsf Q_V^{-1}\textbf{Diag}\Big(\widetilde\lambda_{1}^{W},\cdots,\widetilde\lambda_{d}^W \Big)\mathsf Q_V,\\
\end{cases}
\eea 
where $\mathsf Q_W$ and $\mathsf Q_V$ denote the orthogonal matrix for the eigenvalue decomposition of $\nabla^2_{xy}W(x,y)$ and $\nabla^2_{xx}(W(x,y)+U(x))$ with $\mathsf Q_W^{-1}=\mathsf Q_W^{\ts}$ and $\mathsf Q_V^{-1}=\mathsf Q_V^{\ts}$.
\begin{itemize}
    \item[(1)] 
If there exist positive constants $z_1,z_2,\lambda_{W_{xx}}>0$, such that
\bea\label{final inequality condition}\mathsf A(x,y)\succeq \lambda_{W_{xx}}\mathsf I_{2d}, \quad \text{and}\quad 
C_1< \lambda_{W_{xx}}<C_2,
\eea 
where $\mathsf A(x,y)$ is defined in Definition \ref{definition mean field info matrix}, and 
\beaa
C_1&=&\sqrt{ \frac{z_1^2z_2^2+\frac{z_1^4}{2}+\sqrt{z_1^4z_2^4+z_1^6z_2^2}}{2}}|\lambda_i^W|,\\
C_2&=&\min\Big\{z_1z_2-|\frac{1}{2}[(1+z_1z_2+z_2^2)-z_1^2\widetilde\lambda_i^W]|,\\
&&\qq\qq (1+z_2^2)-z_1z_2\widetilde{\lambda}_i^W	-|		\frac{1}{2}[(1+z_1z_2+z_2^2)-z_1^2\widetilde{\lambda}_i^W]|\Big\},\quad i=1,\cdots, d, 
\eeaa 
then Assumption \eqref{main assumption} holds. 
\item[(2)]  Suppose there exists constants $\overline\lambda\geq\underline\lambda>0$, such that for any $(x,y)\in\mathbb{T}^d\times\mathbb{T}^d$, 
\beaa 
  && \underline\lambda\mathsf{I}_d 	\preceq \nabla^2_{xx}W(x,y)+\nabla^2_{xx}U(x) 	\preceq     \overline\lambda\mathsf{I}_d,
\eeaa 
 Assume that $\nabla^2_{xy}W(x,y)=0$, $z_1=1$, and there exist constant $z_2\in(0,\frac{1+\sqrt{5}}{2})$ and $\delta>0$, such that $\underline{\lambda}$, $\overline{\lambda}$ satisfies the following conditions:
\bea \label{d condition simple}
2\underline{\lambda}-\overline{\lambda}^2>1-\delta,\q [2(z_2-z_2^2)] \underline{\lambda}+
2z_2+2z_2^3-z_2^4-3z_2^2>\delta,
\eea 
then Assumption \ref{main assumption} holds
\end{itemize}
\end{lemma}

\begin{proof}
\noindent \textbf{Case 1:}
According to Definition \ref{definition mean field info matrix}, we have 
\beaa 
\mathfrak R(z,x,y)&=&\frac{1}{2}\begin{pmatrix}
\mathsf{A}(x,y)&\mathsf{B}(x,y)\\
\mathsf{B}(x,y)&\mathsf{A}(y,x)
\end{pmatrix}\\
&=&\frac{1}{2}\begin{pmatrix}
0&\mathsf{B}(x,y)\\
\mathsf{B}(x,y)&0
\end{pmatrix}+\frac{1}{2}\begin{pmatrix}
\mathsf{A}(x,y)&0\\
0&\mathsf{A}(y,x)
\end{pmatrix}\\
&=&\frac{1}{2}(\widetilde{\mathcal J}_1+\widetilde{\mathcal J}_2).
\eeaa 
We want to get a positive lower bound for the spectrum of $\widetilde{J}_2$. Applying the Gershgorin circle theorem. it is sufficient to require the following condition, for $V(x,y)=W(x,y)+U(x)$, 
\beaa
\begin{pmatrix}
	z_1z_2\mathsf I_d& \frac{1}{2}[(1+z_1z_2+z_2^2)\mathsf I_d-z_1^2\nabla^2_{xx}V(x,y)]\\
	\frac{1}{2}[(1+z_1z_2+z_2^2)\mathsf I_d-z_1^2\nabla^2_{xx}V(x,y)]& (1+z_2^2)\mathsf I_d-z_1z_2\nabla^2_{xx} V(x,y)
	\end{pmatrix}-\lambda_{W_{xx}}\mathsf I_{2d}\succeq 0.
	\eeaa  
such that there exists constant $\lambda_{W_{xx}}>0$ satisfying
\beaa 
\widetilde{\mathcal{J}}_2\succeq \lambda_{W_{xx}}\begin{pmatrix}
 \mathsf I_{2d}&0\\
0& \mathsf I_{2d}
\end{pmatrix}.
\eeaa 
According to the eigenvalue decomposition of $\nabla^2_{xx}W(x,y)+\nabla^2_{xx}U(x)$, it is thus sufficient to prove the following inequalities, for $i=1,\cdots,d$,
	\bea \label{Gershgorin inequality}
		z_1z_2-| \frac{1}{2}[(1+z_1z_2+z_2^2)-z_1^2\widetilde\lambda_i^W]|\ge \lambda_{W_{xx}},\nonumber \\
(1+z_2^2)-z_1z_2\widetilde{\lambda}_i^W	-|		\frac{1}{2}[(1+z_1z_2+z_2^2)-z_1^2\widetilde{\lambda}_i^W]|\ge\lambda_{W_{xx}}.
	\eea 
 
Given such a positive $\lambda_{W_{xx}}>0$, we  analyze the first term $\widetilde{\mathcal J}_1$, such that
\bea \label{B + lambda W}
\begin{pmatrix}
0&\mathsf{B}(x,y)\\
\mathsf{B}(x,y)&0
\end{pmatrix}+\lambda_{W_{xx}}\begin{pmatrix}
 \mathsf I_{2d}&0\\
0& \mathsf I_{2d}
\end{pmatrix} \succeq 0.
\eea 
According to Schur complement for symmetric matrix function (see \cite{vandenberghe2004convex}[Appendix A.5]), this is equivalent to the following condition:
\bea\label{schur condition-1}
\begin{cases}
    \lambda_{W_{xx}}>0; \\
    \lambda_{W_{xx}}^2\mathsf I_{2d}-\mathsf B^2(x,y)\succeq 0.
\end{cases}
\eea 
Recall 
\beaa 
\mathsf B(x,y)&=& \begin{pmatrix}
0&-\frac{z_1^2}{2}\nabla^2_{xy}W(x,y)\\
	-\frac{z_1^2}{2}\nabla^2_{xy}W(x,y)&-z_1z_2\nabla^2_{xy}W(x,y)
\end{pmatrix} \\
&=& \begin{pmatrix}
0 & \mathsf Q^{-1}_W\Big[\textbf{Diag}\Big\{\frac{-z_1^2}{2}\lambda_i^W\Big\}_{i=1}^d \Big] \mathsf Q_W\\
\mathsf	Q^{-1}_W\Big[\textbf{Diag}\Big\{\frac{-z_1^2}{2}\lambda_i^W\Big\}_{i=1}^d \Big] \mathsf Q_W&	\mathsf Q^{-1}_W\Big[\textbf{Diag}\Big\{-z_1z_2\lambda_i^W\Big\}_{i=1}^d \Big] \mathsf Q_W
\end{pmatrix}\\
&=& \begin{pmatrix}
0&\mathsf B_{12}\\
\mathsf B_{21}& \mathsf B_{22}
\end{pmatrix}.
\eeaa 
We have
\beaa 
\lambda^2_{W_{xx}}\mathsf I_{2d}-\mathsf B^2(x,y)
&=&\begin{pmatrix}
\lambda^2_{W_{xx}}\mathsf I_d-\mathsf B_{12}^2& -\mathsf B_{12}\mathsf B_{22}\\
-\mathsf B_{12}\mathsf B_{22}&\lambda^2_{W_{xx}}\mathsf I_d-\mathsf B_{22}^2-\mathsf B_{12}^2
\end{pmatrix}=\begin{pmatrix}
\mathsf C_{11}^W& \mathsf C_{12}^W\\
\mathsf C_{12}^W& \mathsf C_{22}^W
\end{pmatrix},
\eeaa 
where we denote 
\beaa 
\mathsf C_{11}^W&=& \mathsf Q_W^{-1}\Big[\textbf{Diag}\Big\{
\lambda^2_{W_{xx}}- \frac{z_1^4}{4}(\lambda_i^W)^2
 \Big\}_{i=1}^d\Big]\mathsf Q_W\\
\mathsf C_{12}^W&=& \mathsf Q_W^{-1}\Big[\textbf{Diag}\Big\{ -z_1z_2\frac{z_1^2}{2}(\lambda_i^W)^2 \Big\}_{i=1}^d\Big]\mathsf Q_W\\
\mathsf C_{22}^W&=& \mathsf Q_W^{-1}\Big[\textbf{Diag}\Big\{ \lambda_{W_{xx}}^2-(\lambda_i^W)^2(z_1^2z_2^2+\frac{z_1^4}{4}) \Big\}_{i=1}^d\Big]\mathsf Q_W.
\eeaa 
Applying Schur complement, it is equivalent to, for $i=1,\cdots, d$,
\bea 
\label{condition 2 U=0}
\begin{cases}
\lambda_{W_{xx}}^2-(\lambda_i^W)^2\frac{z_1^4}{4}\ge 0;\\
[\lambda_{W_{xx}}^2-(\lambda_i^W)^2(z_1^2z_2^2+\frac{z_1^4}{4})][\lambda_{W_{xx}}^2-\frac{z_1^4}{4}(\lambda_i^W)^2]-z_1^2z_2^2\frac{z_1^4}{4}(\lambda_i^W)^4\ge 0.
\end{cases}
\eea 
This is equivalent to 
\bea \label{inequality for lambda W}
\begin{cases} \displaystyle
\frac{\lambda_{W_{xx}}^2}{(\lambda_i^W)^2}-(z_1^2z_2^2+\frac{z_1^4}{4})\ge 0;\\
 \displaystyle\frac{\lambda_{W_{xx}}^4}{(\lambda_i^W)^4}-(z_1^2z_2^2+\frac{z_1^4}{2})\frac{\lambda_{W_{xx}}^2}{(\lambda_i^W)^2}+\frac{z_1^8}{16}\ge 0.
\end{cases}
\eea 
Solving the second inequality, we get
\beaa 
 \frac{\lambda_{W_{xx}}^2}{(\lambda_i^W)^2}\ge  \frac{z_1^2z_2^2+\frac{z_1^4}{2}+\sqrt{z_1^4z_2^4+z_1^6z_2^2}}{2}.
\eeaa 
Notice that 
\beaa
 \frac{z_1^2z_2^2+\frac{z_1^4}{2}+\sqrt{z_1^4z_2^4+z_1^6z_2^2}}{2}>(z_1^2z_2^2+\frac{z_1^4}{4}).
\eeaa 
It is sufficient to prove the following inequality for \eqref{inequality for lambda W}:
\bea\label{est lambda 2}
 \displaystyle  \lambda_{W_{xx}}^2\ge  \frac{z_1^2z_2^2+\frac{z_1^4}{2}+\sqrt{z_1^4z_2^4+z_1^6z_2^2}}{2}(\lambda_i^W)^2.
\eea 
Thus, for $\lambda_{W_{xx}},\lambda_i^W>0$, $i=1,\cdots,d$,
combining with \eqref{Gershgorin inequality}, the matrix $\mathfrak R$ is positive definite, if the following condition holds: 
		\bea
	\begin{cases}
\lambda_{W_{xx}}\le 		z_1z_2-| \frac{1}{2}[(1+z_1z_2+z_2^2)-z_1^2\widetilde\lambda_i^W]|,\\
 \lambda_{W_{xx}}\le (1+z_2^2)-z_1z_2\widetilde{\lambda}_i^W	-|		\frac{1}{2}[(1+z_1z_2+z_2^2)-z_1^2\widetilde{\lambda}_i^W]|,\\
\displaystyle \lambda_{W_{xx}}^2\ge  \frac{z_1^2z_2^2+\frac{z_1^4}{2}+\sqrt{z_1^4z_2^4+z_1^6z_2^2}}{2}(\lambda_i^W)^2,
	\end{cases} 
	\eea
which is condition \eqref{final inequality condition}. The proof is completed.

\noindent\textbf{Case 2:}
We apply the Schur complement for symmetric matrix function $\mathfrak{R}$. The following conditions are equivalent. 
\begin{itemize}
\item[(1)]
$\mathfrak R\succeq 0$ ($\mathfrak R$ is positive definite).
\item[(2)]$\mathsf A(x,y)\succeq 0$, $(\mathsf I_{2d}-\mathsf A(x,y)\mathsf A^{-1}(x,y))\mathsf B(x,y)=0$, $\mathsf A(y,x)-\mathsf B(x,y) \mathsf A^{-1}(x,y)\mathsf B(x,y)\succeq 0$.
\end{itemize}
According to our assumption $\nabla^2_{xy}W=0$, thus $\mathsf B=0$. We only need to show that $\mathsf A(x,y)$ in positive definite. Similar arguments then apply to $\mathsf A(y,x)$. Denote 
\beaa 
\mathsf A(x,y)=\begin{pmatrix}
\mathsf A_{11}&\mathsf A_{12}\\
\mathsf A_{21}&\mathsf A_{22}
\end{pmatrix},
\eeaa 
with 
\bea\label{block for A} 
\mathsf A_{11} &=& z_1z_2\mathsf I_d,\quad  \mathsf A_{12}=\mathsf A_{21}=\frac{1}{2}[(1+z_1z_2+z_2^2)\mathsf I_d-z_1^2\nabla^2_{xx}V(x,y) ],\\
\mathsf A_{22}&=&(1+z_2^2)\mathsf I_d-z_1z_2\nabla^2_{xx} V(x,y).
\eea 
Applying the Schur complement for symmetric matrix $\mathsf A$, it is equivalent to find $z_1$,  $z_2\in\mathbb{R}^1$, such that
\begin{equation*}
    z_1z_2>0,\qquad 1+z_2^2-\overline{\lambda}z_1z_2>0,
\end{equation*}
and 
 \beaa 
 \mathsf A_{22}-\mathsf A_{12}\mathsf A_{11}^{-1}\mathsf A_{21}\succeq 0 \Longleftrightarrow  z_1z_2A_{22}-\mathsf A_{12}^2 \succeq0.
 \eeaa 
By direct computation, it is equivalent to the following condition,
\beaa
z_1z_2[(1+z_2^2)\mathsf I_{d}-z_1z_2 \nabla^2_{xx}V(x,y)] -\frac{1}{4}((1+z_1z_2+z_2^2)\mathsf I_{d}-z_1^2 \nabla^2_{xx}V(x,y))^2\succeq 0.
\eeaa
By a direct computation, it is equivalent to the following inequality:
\beaa 
-z_1^4(\nabla_{xx}^2V)^2+[2(1+z_1z_2-z_2^2)z_1^2] \nabla_{xx}^2V+[2
z_1z_2+2z_1z_2^3-1-(z_1^2+2)z_2^2-z_2^4]\mathsf I_d>0.
\eeaa 
Based on the assumption of $\nabla^2_{xx}W(x,y)+\nabla^2_{xx}U(x)=\nabla^2_{xx}V(x,y)$, we have 
\beaa 
\nabla^2_{xx}V(x,y)&=&\mathsf Q_V^{-1}\textbf{Diag}(\widetilde\lambda_1^W,\cdots,\widetilde\lambda_d^W)\mathsf Q_V.
\eeaa 
In particular, we assume $0<\underline{\lambda} \le\widetilde \lambda_1^W\le\widetilde\lambda^W_2\le \cdots\le\widetilde \lambda_d^W\le \overline{\lambda}$, where $\widetilde\lambda_1^W,\cdots,\widetilde\lambda_d^W$ are 
eigenvalues of matrix $\nabla^2_{xx} W(x,y)+\nabla^2_{xx} U(x)$, and  $\underline\lambda$, $\overline\lambda$ are lower bound and upper bound of these eigenvalues, respectively. Applying the lower and upper bound of the eigenvalues, it is sufficient to prove the following conditions: 
\bea\label{1d condition}
\begin{cases}
& z_1z_2>0,\q  (1+z_1z_2-z_2^2)>0;\\
&-z_1^4\overline{\lambda}^2+[2(1+z_1z_2-z_2^2)z_1^2] \underline{\lambda}+
2z_1z_2+2z_1z_2^3-1-(z_1^2+2)z_2^2-z_2^4>0.
\end{cases}
\eea
Let $z_1=1$, then \eqref{d condition simple} implies \eqref{1d condition}. Let $1+z_2-z_2^2>0$, for a small constant $\delta>0$, there exists $z_2>0$ such that $[2(z_2-z_2^2)] \underline{\lambda}+
2z_2+2z_2^3-z_2^4-3z_2^2>\delta$, which completes the proof for condition (1).
\qed 
\end{proof}

\begin{proof}[Proof of Example \ref{ex3}]
We provide a simple proof of Example \ref{ex3} by applying Condition (2) in Lemma \ref{lemma: W x - y}. If $\underline{\lambda}\mathsf I_{d}\preceq\nabla^2_{xx}(U(x)+W(x,y))\preceq\overline{\lambda}\mathsf I_d$, which satisfies the condition(2) in Lemma \ref{lemma: W x - y}, then as long as 
{ $\|\nabla^2_{x,y} W(x,y)\|_{\mathrm{F}}=O(\varepsilon)$  is small enough for $\varepsilon>0$}, matrix $\mathfrak R$ remains positive definite. In particular, if we pick $z_1=1$, $z_2=0.3$, $\delta=0.02$, this implies
\beaa 
2\underline{\lambda}-\overline{\lambda}^2>1-\delta,\q [2(z_2-z_2^2)] \underline{\lambda}+
2z_2+2z_2^3-z_2^4-3z_2^2>0.02,
\eeaa 
i.e. $2\underline{\lambda}-\overline{\lambda}^2>0.08$, and $0.42\underline{\lambda}+0.3759>0.02$. Hence we require that $2\underline{\lambda}-\overline{\lambda}^2>0.08$.
\end{proof}

\subsection{Proof of Example \ref{ex2}}
Similar to the previous section, we apply Schur complement to derive the positive definite condition for $\mathfrak R$.
In particular, we rewrite matrix $\mathfrak R$ in the following form, where we separate the potential function $U(\cdot)$ and the interacting potential function $W(x,y)$. 
\begin{definition}[Reformulation of Mean--field information matrix]\label{definition mean field info matrix-2}
Define a symmetric matrix function $\mathfrak R\in\mathbb{R}^{4d\times 4d}$, such that
\beaa 
\mathfrak R(z,x,y)=\frac{1}{2}\begin{pmatrix}
\mathsf{A}_1(x,y)&\mathsf{B}(x,y)\\
\mathsf{B}(x,y)&\mathsf{A}_1(y,x)
\end{pmatrix}+\frac{1}{2}\begin{pmatrix}
\mathsf A_2(x,y)&0\\
0&\mathsf A_2(y,x)
\end{pmatrix},
\eeaa 
where $\mathsf{A}_1$, $\mathsf A_2$, and $\mathsf{B}\in\mathbb{R}^{2d\times 2d}$ 
are defined below:
\beaa
\mathsf A_1(x,y)&=& \begin{pmatrix}
	0& -\frac{z_1^2}{2}\nabla^2_{xx}W(x,y)  \\
	-\frac{z_1^2}{2}\nabla^2_{xx}W(x,y)& -z_1z_2\nabla^2_{xx} W(x,y)
\end{pmatrix}, \\
\mathsf A_1(y,x)&=& \begin{pmatrix}
	0& -\frac{z_1^2}{2}\nabla^2_{yy}W(y,x)  \\
	-\frac{z_1^2}{2}\nabla^2_{yy}W(y,x)& -z_1z_2\nabla^2_{yy} W(y,x)
\end{pmatrix}, \\
\mathsf B(x,y)&=& \begin{pmatrix}
0&-\frac{z_1^2}{2}\nabla^2_{xy}W(x,y)\\
	-\frac{z_1^2}{2}\nabla^2_{xy}W(x,y)&-z_1z_2\nabla^2_{xy}W(x,y)
\end{pmatrix},
\eeaa
and 
\beaa
\mathsf A_2(x,y)&=&\begin{pmatrix}
	z_1z_2\mathsf I_d& \frac{1}{2}[(1+z_1z_2+z_2^2)\mathsf I_d-z_1^2\nabla^2_{xx}U(x)]\\
	\frac{1}{2}[(1+z_1z_2+z_2^2)\mathsf I_d-z_1^2\nabla^2_{xx}U(x)]& (1+z_2^2)\mathsf I_d-z_1z_2\nabla^2_{xx} U(x)
	\end{pmatrix},\\
	\mathsf A_2(y,x)&=&\begin{pmatrix}
	z_1z_2\mathsf I_d& \frac{1}{2}[(1+z_1z_2+z_2^2)\mathsf I_d-z_1^2\nabla^2_{yy}U(y)]\\
	\frac{1}{2}[(1+z_1z_2+z_2^2)\mathsf I_d-z_1^2\nabla^2_{yy}U(y)]& (1+z_2^2)\mathsf I_d-z_1z_2\nabla^2_{yy} U(y)
	\end{pmatrix}.
	\eeaa
\end{definition}

\begin{lemma}\label{lemma: W x - y-2}
    Assume that 
\bea\label{assump: hesss W xy-2}
\begin{cases}
\nabla^2_{xy}W(x,y)=\mathsf Q^{-1}_W\textbf{Diag}\Big(\lambda_1^W,\cdots,\lambda_d^W \Big)\mathsf Q_W,\\
\nabla^2_{xx}W(x,y)= \mathsf Q^{-1}_W\textbf{Diag}\Big(\widetilde\lambda_{1}^{W},\cdots,\widetilde\lambda_{d}^W \Big)\mathsf Q_W,\\
\underline\lambda\mathsf{I}_d 	\preceq \nabla^2_{xx}U(x) 	\preceq     \overline\lambda\mathsf{I}_d.
\end{cases}
\eea 
And there exists $z_1>0$ and $z_2>0$, such that 
\beaa 
\mathsf A_2\succeq \lambda_{U}\mathsf I_{2d},
\eeaa 
where we denote $\lambda_{U}$ as the spectrum lower bound for $\mathsf A_2$ defined in Definition \ref{definition mean field info matrix-2}. Then if $\lambda_i^W$ sufficiently small, and
\bea \label{condition on eigencalue interval}
\widetilde\lambda_i^W\in (\frac{-2\lambda_{U}(z_2+\sqrt{z_1^2+z_2^2})}{z_1^3}, \frac{2\lambda_{U}(\sqrt{z_1^2+z_2^2}-z_2)}{z_1^3}),\quad i=1,\cdots, d, 
\eea 
then Assumption \eqref{main assumption} holds.
\end{lemma}

\begin{proof}
According to Definition \ref{definition mean field info matrix-2}, we have 
\beaa 
\mathfrak R(z,x,y)&=&\frac{1}{2}\begin{pmatrix}
\mathsf{A}_1(x,y)&\mathsf{B}(x,y)\\
\mathsf{B}(x,y)&\mathsf{A}_1(y,x)
\end{pmatrix}+\frac{1}{2}\begin{pmatrix}
\mathsf A_2(x,y)&0\\
0&\mathsf A_2(y,x)
\end{pmatrix}\\
&=&\frac{1}{2}(\mathcal J_1+\mathcal J_2).
\eeaa 
According to condition \eqref{Gershgorin inequality} in the proof of Lemma \ref{lemma: W x - y}, we find a sufficient condition such that $\mathcal J_2$ is positive definite, i.e. there exists constant $\lambda_{U}>0$, such that 
\beaa 
\mathcal{J}_2\succeq \lambda_{U}\begin{pmatrix}
 \mathsf I_{2d}&0\\
0& \mathsf I_{2d}
\end{pmatrix}.
\eeaa 
Next, we  analyze the first term $\mathcal J_1$, such that $\mathcal J_1+\mathcal J_2$ is positive definite, which implies that $\mathcal J_1$ could be potentially non-positive definite. We shall show that
\beaa 
\begin{pmatrix}
\mathsf{A}_1(x,y)&\mathsf{B}(x,y)\\
\mathsf{B}(x,y)&\mathsf{A}_1(y,x)
\end{pmatrix}+\begin{pmatrix}
\lambda_{U} \mathsf I_{2d}&0\\
0&\lambda_{U} \mathsf I_{2d}
\end{pmatrix} \succeq 0.
\eeaa 
According to Schur complement, this is equivalent to the following condition:
\bea\label{schur condition}
\begin{cases}
    \mathsf A_1(x,y)+\lambda_{U}\mathsf I_{2d}\succeq 0; \\
    \mathsf A_1(y,x)+\lambda_{U}\mathsf I_{2d}-\mathsf B(x,y)[\mathsf A_1(x,y)+\lambda_{U}\mathsf I_{2d}]^{-1}\mathsf B(x,y)\succeq 0.
\end{cases}
\eea 
Based on Assumption \ref{assump: hesss W xy-2}, the first condition is represented below:  
\beaa 
&&\lambda_{U}\begin{pmatrix}
\mathsf I_d&0\\
0&\mathsf I_d
\end{pmatrix} +\mathsf A_1(y,x)\\
&=& \begin{pmatrix}
	\lambda_{U}\mathsf I_d &
	\mathsf Q^{-1}_W\Big[\textbf{Diag}\Big\{-\frac{z_1^2}{2}\widetilde\lambda_i^W\Big\}_{i=1}^d \Big] \mathsf Q_W  \\
	\mathsf Q^{-1}_W\Big[\textbf{Diag}\Big\{-\frac{z_1^2}{2}\widetilde\lambda_i^W\Big\}_{i=1}^d \Big] \mathsf Q_W & 	\mathsf Q^{-1}_W\Big[\textbf{Diag}\Big\{\lambda_{U} -z_1z_2\widetilde\lambda_i^W\Big\}_{i=1}^d \Big] \mathsf Q_W 
\end{pmatrix}\succeq 0.
\eeaa 
Thus, the first condition in \eqref{schur condition} is equivalent to 
\bea \label{condition 1}
\lambda_{U}(\lambda_{U}-z_1z_2\widetilde\lambda_i^W)-\frac{z_1^4}{4}(\widetilde\lambda_i^W)^2>0, \quad \text{for}\quad i=1,\cdots,d,
\eea 
which implies
\bea \label{condition on eigencalue interval-2}
\widetilde\lambda_i^W\in (\frac{-2\lambda_{U}(z_2+\sqrt{z_1^2+z_2^2})}{z_1^3}, \frac{2\lambda_{U}(\sqrt{z_1^2+z_2^2}-z_2)}{z_1^3}).
\eea 
Similarly, we have 
\beaa 
\mathsf B(x,y)&=& \begin{pmatrix}
0 &
	\mathsf Q^{-1}_W\Big[\textbf{Diag}\Big\{-\frac{z_1^2}{2}\lambda_i^W\Big\}_{i=1}^d \Big] \mathsf Q_W  \\
	\mathsf Q^{-1}_W\Big[\textbf{Diag}\Big\{-\frac{z_1^2}{2}\lambda_i^W\Big\}_{i=1}^d \Big] \mathsf Q_W & 	\mathsf Q^{-1}_W\Big[\textbf{Diag}\Big\{ -z_1z_2\lambda_i^W\Big\}_{i=1}^d \Big] \mathsf Q_W 
\end{pmatrix}\\
&=& \begin{pmatrix}
0&\mathsf B_{12}\\
\mathsf B_{21}&\mathsf B_{22}
\end{pmatrix}.
\eeaa 
Since each block of matrix $\mathsf A_1(x,y)+\lambda_{U}\mathsf I_{2d}$ is diagonal, we have 
\beaa 
&&[\mathsf A_1(x,y)+\lambda_{U}\mathsf I_{2d}]^{-1}=\begin{pmatrix}
\mathsf A_{11}^{W}&\mathsf A_{12}^{W}\\
\mathsf A_{21}^{W}&\mathsf A_{22}^{W}
\end{pmatrix},
\eeaa 
where we denote 
\beaa 
\mathsf D_i=\lambda_{U}(\lambda_{U}-z_1z_2\widetilde\lambda_i^W)-\frac{z_1^4}{4}(\widetilde\lambda_i^W)^2,\quad \text{for}\quad i=1,\cdots,d,
\eeaa 
and
\beaa
\mathsf A_{11}^W&=&\mathsf Q_W^{-1}\Big[\textbf{Diag}\Big(\Big\{\frac{1}{\mathsf D_i}[\lambda_{U}-z_1z_2\widetilde\lambda_i^W  ] \Big\}_{i=1}^d \Big) \Big] \mathsf Q_W\\
\mathsf A_{12}^W&=&\mathsf A_{21}^W=\mathsf Q_W^{-1}\Big[\textbf{Diag}\Big(\Big\{ \frac{z_1^2\widetilde\lambda_i^W}{2\mathsf D_i} \Big\}_{i=1}^d \Big) \Big] \mathsf Q_W\\
\mathsf A_{22}^W&=&\mathsf Q_W^{-1}\Big[\textbf{Diag}\Big(\Big\{\frac{\lambda_{U}}{\mathsf D_i} \Big\}_{i=1}^d \Big) \Big] \mathsf Q_W.
\eeaa
Using the above explicit representation, we obtain
\beaa 
&&\mathsf A_1(y,x)+\lambda_{U}\mathsf I_{2d}-\mathsf B(x,y)[\mathsf A(x,y)+\lambda_{U}\mathsf I_{2d}]^{-1}\mathsf B(x,y)\\
&=&A_1(y,x)+\lambda_{U}\mathsf I_{2d}\\
&&-\begin{pmatrix}
\mathsf B_{12}\mathsf A_{22}^W\mathsf B_{21}&\mathsf B_{12}\mathsf A_{21}^W\mathsf B_{12}+\mathsf B_{12}\mathsf A_{22}^W\mathsf B_{22}\\
\mathsf B_{21}\mathsf A_{12}^W\mathsf B_{21}+\mathsf B_{22}\mathsf A_{22}^W\mathsf B_{21}&\mathsf B_{21}\mathsf A^W_{11}\mathsf B_{12}+\mathsf B_{22}\mathsf A_{21}^W\mathsf B_{12}+\mathsf B_{21}\mathsf A_{12}^W\mathsf B_{22}+\mathsf B_{22}\mathsf A_{22}^W\mathsf B_{22}
\end{pmatrix} \\ 
&=&\begin{pmatrix}
\widetilde{\mathsf C}_{11}& \widetilde{\mathsf C}_{12}\\
\widetilde{\mathsf C}_{21}&\widetilde{\mathsf C}_{22}
\end{pmatrix},
\eeaa 
where we denote 
\beaa 
\widetilde{\mathsf C}_{11}&=&\mathsf Q_W^{-1}\Big[\textbf{Diag}\Big(\Big\{\lambda_{U}- \Big[ \frac{z_1^4\lambda_U(\lambda_i^W)^2}{4\mathsf D_i} \Big] \Big\}_{i=1}^d \Big) \Big] \mathsf Q_W,\\
\widetilde{\mathsf C}_{12}&=& \mathsf Q_W^{-1}\Big[\textbf{Diag}\Big(\Big\{-\frac{z_1^2}{2}\widetilde\lambda_i^W-\Big[ \frac{z_1^3z_2(\lambda_i^W)^2\lambda_U}{2\mathsf D_i} +\frac{z_1^6(\lambda_i^W)^2\widetilde\lambda_i^W}{8\mathsf D_i}\Big]  \Big\}_{i=1}^d \Big) \Big] \mathsf Q_W,\\
\widetilde{\mathsf C}_{22}&=& \mathsf Q_W^{-1}\Big[\textbf{Diag}\Big(\Big\{\lambda_{U} -z_1z_2\widetilde\lambda_i^W-\Big[\frac{z_1^4(\lambda_i^W)^2(\lambda_{U}-z_1z_2\widetilde\lambda_i^W  )}{4\mathsf D_i}\\
&&\qq\qq \qq  +\frac{z_1^5z_2(\lambda_i^W)^2\widetilde\lambda_i^W}{2\mathsf D_i}+\frac{\lambda_{U}(z_1z_2)^2(\lambda_i^W)^2}{\mathsf D_i}\Big]  \Big\}_{i=1}^d \Big) \Big] \mathsf Q_W.
\eeaa 
Applying Schur complement and under the condition $\mathsf D_i>0$ for $i=1,\cdots,d$ (i.e. \eqref{condition on eigencalue interval-2}), the second condition in \eqref{schur condition} is equivalent to, for $i=1,\cdots,d$,
\bea 
\label{condition 2}
\begin{cases}
4\lambda_{U}\mathsf D_i- z_1^4\lambda_U(\lambda_i^W)^2>0,\\
4 [4\lambda_{U}\mathsf D_i- z_1^4\lambda_U(\lambda_i^W)^2]\times \\
[4\mathsf D_i(\lambda_{U} -z_1z_2\widetilde\lambda_i^W)-[z_1^4(\lambda_i^W)^2(\lambda_{U}-z_1z_2\widetilde\lambda_i^W  ) +2z_1^5z_2(\lambda_i^W)^2\widetilde\lambda_i^W+4\lambda_{U}(z_1z_2)^2(\lambda_i^W)^2]]\\
-[4\mathsf D_iz_1^2\widetilde\lambda_i^W+ 4z_1^3z_2(\lambda_i^W)^2\lambda_U +z_1^6(\lambda_i^W)^2\widetilde\lambda_i^W ]^2>0.
\end{cases}
\eea 
Recall that $
\mathsf D_i=\lambda_{U}(\lambda_{U}-z_1z_2\widetilde\lambda_i^W)-\frac{z_1^4}{4}(\widetilde\lambda_i^W)^2,\quad \text{for}\quad i=1,\cdots,d.$ 
For the first inequality in \eqref{condition 2}, it is sufficient to require that, for $(\lambda_i^W)^2$ small enough,
\beaa
\mathsf D_i=\lambda_{U}(\lambda_{U}-z_1z_2\widetilde\lambda_i^W)-\frac{z_1^4}{4}(\widetilde\lambda_i^W)^2>0.
\eeaa
For the second inequality in \eqref{condition 2}, denote $\lambda_{U}-z_1z_2\widetilde\lambda_i^W=\alpha$. We observe the following simplification:
\bea\label{last bound for W xy} 
&&4[4\mathsf D_i\lambda_U- z_1^4(\lambda_i^W)^2\lambda_U]
\times\nonumber \\
&&\Big[4\mathsf D_i\alpha-\Big(z_1^4(\lambda_i^W)^2\alpha+2z_1^5z_2(\lambda_i^W)^2\widetilde\lambda_i^W+4\lambda_U(z_1z_2)^2(\lambda_i^W)^2\Big)\Big]\nonumber\\
&&- [4\mathsf D_iz_1^2\widetilde\lambda_i^W+ 4z_1^3z_2(\lambda_i^W)^2\lambda_U +z_1^6(\lambda_i^W)^2\widetilde\lambda_i^W ]^2>0,\nonumber\\
&\Rightarrow&4\lambda_U[4\mathsf D_i- \varepsilon_1]
\times\Big[4\mathsf D_i\alpha-\varepsilon_2\Big]- \Big[4\mathsf D_iz_1^2\widetilde\lambda_i^W+\varepsilon_3\Big]^2>0,\nonumber\\
&\Rightarrow& 64\alpha \mathsf D_i^2\lambda_{U}-16\mathsf D_i^2z_1^4(\widetilde\lambda_i^W)^2+ o(\varepsilon)\nonumber>0,\\
&\Rightarrow& 64 \mathsf D_i^2 (\lambda_{U}\alpha-\frac{1}{4}z_1^4(\widetilde\lambda_i^W)^2)+ o(\varepsilon)>0.
\eea 
Here all $\varepsilon$ terms depend on $(\lambda_i^{W})^2$. And the leading term $(\lambda_{U}\alpha-\frac{1}{4}z_1^4(\widetilde\lambda_i^W)^2)>0$ is the same as \ref{condition 1}. Thus we finish the proof as long as $\varepsilon$ is small enough.
\qed 
\end{proof}

\begin{proof}[Proof of Example \ref{ex2}]
According to the above Lemma, let $\underline\lambda=\overline{\lambda}=0.9$, $z_1=1$, and $z_2=0.3$, we obtain $\lambda_U\approx0.2$. By direct computations, condition \eqref{condition on eigencalue interval} implies 
\beaa 
-0.4(0.3+\sqrt{1.09})< \widetilde\lambda_i^W<  0.4(\sqrt{1.09}-0.3),\quad i=1,\cdots, d.
\eeaa 
In particular, if we choose $\widetilde\lambda_i^W=-0.12$, plugging into \eqref{last bound for W xy}, it is sufficient to require $\lambda_i^W<10^{-3}$, such that $o(\varepsilon)$ is small enough. \qed 
\end{proof}

\begin{proof}[Proof of Remark \ref{rem 3}]
For $\mathsf U\equiv 0$, and $W(x,y)=W(x-y)$, we have $\nabla^2_{xx}W(x-y)=-\nabla^2_{xy}W(x-y)$ and $\lambda_i^W=-\widetilde\lambda_i^W$. According to the eigenvalue decompostion, we have
 \beaa 
 \begin{pmatrix}
     \mathsf A&\mathsf B\\
     \mathsf B&\mathsf A
 \end{pmatrix}=\textbf{Diag}\Big(\mathsf Q_W^{-1},\mathsf Q_W^{-1},\mathsf Q_W^{-1},\mathsf Q_W^{-1} \Big) \begin{pmatrix}
     \widehat{\mathsf A} &\widehat{\mathsf B}\\
     \widehat{\mathsf B} &\widehat{\mathsf A}
 \end{pmatrix}\textbf{Diag}\Big(\mathsf Q_W,\mathsf Q_W,\mathsf Q_W,\mathsf Q_W \Big),
 \eeaa 
 where 
 \beaa 
 \widehat{\mathsf A}&=&\begin{pmatrix}
	z_1z_2\mathsf I_d& \frac{1}{2}[(1+z_1z_2+z_2^2)\mathsf I_d-z_1^2\textbf{Diag}\Big\{\widetilde\lambda_i^W\Big\}_{i=1}^d]\\
	\frac{1}{2}[(1+z_1z_2+z_2^2)\mathsf I_d-z_1^2\textbf{Diag}\Big\{\widetilde\lambda_i^W\Big\}_{i=1}^d]& (1+z_2^2)\mathsf I_d-z_1z_2\textbf{Diag}\Big\{\widetilde\lambda_i^W\Big\}_{i=1}^d]
	\end{pmatrix},\\
	\widehat{\mathsf B}&=&\begin{pmatrix}
	    0& \textbf{Diag}\Big\{-\frac{z_1^2}{2}\lambda_i^W\Big\}_{i=1}^d\\
	    \textbf{Diag}\Big\{-\frac{z_1^2}{2}\lambda_i^W\Big\}_{i=1}^d&  \textbf{Diag}\Big\{-z_1z_2\lambda_i^W\Big\}_{i=1}^d
	\end{pmatrix}.
 \eeaa 
For $\mathfrak R$ positive definite, it is equivalent to the following condition: 
\beaa 
\begin{pmatrix}
     \widehat{\mathsf A} &\widehat{\mathsf B}\\
     \widehat{\mathsf B} &\widehat{\mathsf A}
 \end{pmatrix}\succeq \lambda \mathsf I_{4d},
\eeaa 
for some positive constant $\lambda>0$. It is thus sufficient to prove the following inequalities, for $i=1,\cdots,d$, and $z_1>0, z_2>0$,
	\bea \label{Gershgorin inequality: x-y case}
		z_1z_2-| \frac{1}{2}[(1+z_1z_2+z_2^2)-z_1^2\widetilde\lambda_i^W]| - \frac{z_1^2}{2}|\lambda_i^W|\ge \lambda ,\nonumber \\
(1+z_2^2)-z_1z_2\widetilde{\lambda}_i^W	-|		\frac{1}{2}[(1+z_1z_2+z_2^2)-z_1^2\widetilde{\lambda}_i^W]|-\frac{z_1^2}{2}|\lambda_i^W|-z_1z_2|\lambda_i^W|\ge\lambda.
	\eea 
Assume that $\widetilde\lambda_i^W=-\lambda_i^W>0$, and $\frac{1}{2}[(1+z_1z_2+z_2^2)-z_1^2\widetilde\lambda_i^W]>0$, then the above inequality is reduced to the following inequalities: \bea\label{neg lower bound}
\frac{z_1z_2-1-z_2^2}{2}\ge \lambda,
\quad -\frac{z_1z_2-1-z_2^2}{2}-2z_1z_2\widehat\lambda_i^W\ge \lambda.
\eea 
It is obvious that such a constant $\lambda$ does not exist unless $z_1z_2-1-z_2^2=0$, $\lambda=0$, and $\widetilde{\lambda} 
_i^W=0$, for $i=1,\cdots,d$. This implies that exponential decay does not hold for $U\equiv 0$ and $W(x,y)=W(x-y)$. In this case, the $O(t^{-\infty})$ convergence derived in \cite{villani2006}[Theorem 56] seems to be optimal. Meanwhile, from \eqref{neg lower bound}, we can provide an estimate of the negative lower bound for eigenvalues of $\mathfrak{R}$. 
\end{proof}

\begin{remark}
Following the above remark, for $\mathsf U\equiv 0$, and $W(x,y)=W(x-y)$, we have $\mathsf A_1(y,x)=\mathsf A_1(x,y)=-\mathsf B(x,y)$. Similar to the condition \eqref{neg lower bound}, matrix $\mathsf A_2$ is at most semi-positive definite. Let $z_1z_2-1-z_2^2=0$, then $\mathsf A_2\succeq 0$. We also note that 
\bea \label{villani mccan carrilo condition}
&&\begin{pmatrix}
\phi_1(x,v) & \phi_2(x,v)& \phi_1(y,\tilde v)& \phi_2(y,\tilde v)
\end{pmatrix}\mathfrak R  \begin{pmatrix}
\phi_1(x,v) & \phi_2(x,v)& \phi_1(y,\tilde v)& \phi_2(y,\tilde v)
\end{pmatrix}^{\ts}\nonumber \\
&\succeq&\frac{1}{2} \begin{pmatrix}
    \phi_1(x,v)-\phi_1(y,\tilde v)&\phi_2(x,v)-\phi_2(y,\tilde v)
\end{pmatrix} \mathsf A_1(x,y)\begin{pmatrix}
    \phi_1(x,v)-\phi_1(y,\tilde v)\\
    \phi_2(x,v)-\phi_2(y,\tilde v)
\end{pmatrix}.
\eea 
The above condition \eqref{villani mccan carrilo condition} recovers the similar Hessian matrix for non-degenerate gradient flow equations in \cite{carrillo2003kinetic}. However, matrix $\mathsf A_1$ is always negative definite even if we assume that $\nabla^2 W$ is positive definite. These facts show major differences between degenerate and non-degenerate gradient flow equations. \end{remark}

\section{Proofs of Theorem \ref{main theorem} and Corollary \ref{cor 2}}\label{section4}

In this section, we present the main proof in this paper. For simplicity of presentation, we denote $f=f(t,x,v)$ as the solution of PDE \eqref{Kinetic FP}. 

We rewrite \eqref{Kinetic FP} in the following equivalent form: 
\bea\label{kinetic FP gamma}
	\pa_tf = \nabla_{x,v}\cdot  ( f \gamma )+\nabla_{x,v}\cdot (f aa^{\ts}\nabla_{x,v}\frac{\delta}{\delta f}\mathcal{E}(f)),
	\eea
where  
\bea\label{gamma}
\gamma=\mathsf J\nabla_{x,v}[\frac{v^2}{2}+\int_{\T^{d}\times\hR^d} W(x,y)f(t,y,\widetilde v)d\widetilde vdy+U(x)],
\eea 
and 
\bea\label{matrix J} 
 \mathsf J=\begin{pmatrix}
	0&-\mathsf I_d\\
	\mathsf I_d&0
\end{pmatrix}_{2d\times 2d}.
\eea
Formulation \eqref{kinetic FP gamma} is known as the flux-gradient flow \cite{li2021controlling} or Pre-Generic \cite{duong2021nonreversible}. We summarize several lemmas below.

\begin{proposition}\label{prop: decomposition of FP}
For $\gamma$ and $\mathsf J$ defined in \eqref{gamma} and \eqref{matrix J}, the PDE \eqref{Kinetic FP} is equivalent to the following form
	\bea 
\pa_t f =\nabla_{x,v}\cdot (f\gamma  )+\nabla\cdot (f aa^{\ts}\nabla_{x,v}\frac{\delta}{\delta f}\mathcal E(f)).
\eea 
Furthermore, the following identity holds:
\begin{equation}\label{ia} 
	\nabla_{x,v}\cdot (f\gamma)=\nabla_{x,v}\cdot (f\mathsf J\nabla \frac{\delta}{\delta f}\mathcal E(f))=f\la \nabla_{x,v} \frac{\delta}{\delta f}\mathcal E(f),\gamma\ra.
\end{equation}
In particular, we denote 
\bea 
\label{w rho convulution general}
W\oast\rho=\int_{\T^d\times\hR^d}W(x,y)f(t,y,\tilde v)dyd\tilde v,\quad \rho(t,y)=\int_{\hR^d}f(t,y,\tilde v)d\Tilde{v}.
\eea 
\end{proposition}
\begin{proof}
First, we observe that 
\beaa 
&&\begin{pmatrix}
	v\\
	-(\int_{\T^{d}\times\hR^d}\nabla_x W(x,y)f(t,y,\widetilde v)d\widetilde vdy+\nabla_x U(x))
\end{pmatrix}\\
&=&-\mathsf J\nabla_{x,v}\left[\frac{v^2}{2}+ \int_{\T^{d}\times\hR^d} W(x,y)f(t,y,\widetilde v)d\widetilde vdy+U(x)\right].
\eeaa 
Furthermore, using the fact $\nabla\cdot (\mathsf J \nabla \phi)=0$ for any smooth function $\phi$, \eqref{Kinetic FP} is equivalent to the following formulation:
\beaa  
&&\pa_t f -\nabla_{x,v} \cdot \Big(f \mathsf J\nabla_{x,v} \Big[\frac{v^2}{2}+\int_{\T^{d}\times\hR^d} W(x,y)f(t,y,\widetilde v)d\widetilde vdy +U(x)\Big]\Big)\\
&=&\nabla_{x,v}\cdot \Big(faa^{\ts}\nabla_{x,v}\Big[\log f+ \frac{v^2}{2}+\int_{\T^{d}\times\hR^d} W(x,y)f(t,y,\widetilde v)d\widetilde vdy +U(x)\Big]\Big).
\eeaa  
Applying the fact that $\frac{\delta}{\delta f}\mathcal E=[\log f+1+ \frac{v^2}{2}+\int_{\T^{d}\times\hR^d} W(x,y)f(t,y,\widetilde v)d\widetilde vdy+U(x)]$, we have 
\beaa
\pa_t f -\nabla_{x,v}\cdot (f\gamma  )=\nabla\cdot (f aa^{\ts}\nabla_{x,v}\frac{\delta}{\delta f}\mathcal E(f)),
\eeaa 
where $\gamma=\mathsf J\nabla_{x,v}[\frac{v^2}{2}+\int_{\T^{d}\times\hR^d} W(x,y)f(t,y,\widetilde v)d\widetilde vdy+U(x)]$. Furthermore, we observe that
\beaa 
\nabla_{x,v}\cdot (f\gamma)
&=&f\la \nabla_{x,v}\log f,\mathsf J \nabla_{x,v}[\frac{v^2}{2}+\int_{\T^{d}\times\hR^d} W(x,y)f(t,y,\widetilde v)d\widetilde vdy+U(x)]\ra \\
&=&f \la \nabla_{x,v}\frac{\delta}{\delta f}\mathcal E(f),\gamma\ra,
\eeaa 
where we add $\la \nabla_{x,v}[\frac{v^2}{2}+U(x)+W\oast\rho(x) ],\mathsf J \nabla_{x,v}[\frac{v^2}{2}+U(x)+W\oast\rho(x)]\ra=0$ in the last step. Notice that
$\nabla\cdot (f\mathsf J \nabla \log f)=0$, we obtain 
\begin{equation*}
\nabla_{x,v}\cdot (f\gamma)=\nabla\cdot (f\mathsf J\nabla_{x,v}\frac{\delta}{\delta f}\mathcal E(f)).
\end{equation*}
\qed 
\end{proof}

\begin{lemma}\label{lemma energy decay}
\label{lemma kl decay} Under the assumption $W(x,y)=W(y,x)$, we have 
\beaa
\frac{\pa}{\pa t} \mathcal{E}(f)=-\int_{\Omega} (\nabla \frac{\delta}{\delta f}\mathcal{E}(f), aa^{\ts}\nabla \frac{\delta}{\delta f}\mathcal{E}(f)) f dxdv=-\mathcal{DE}_a(f),
\eeaa
where $\cE(f)$ is the free energy defined in \eqref{free energy}, and $\mathcal{DE}_a(f)$ is defined in \eqref{fisher}.
\end{lemma}
\begin{proof}
Clearly, PDE \eqref{Kinetic FP} in formulation \eqref{kinetic FP gamma} can be written below. From equality \eqref{ia}, we have 
\begin{equation*}
    \pa_t f =\nabla_{x,v}\cdot (f (aa^{\ts}+\mathsf J)\nabla_{x,v}\frac{\delta}{\delta f}\mathcal E(f)).
\end{equation*}
Thus 
\begin{equation*}
\begin{split}
\frac{d}{dt}\mathcal{E}(f)=&\int_\Omega \frac{\delta}{\delta f}\mathcal{E}(f)\cdot \partial_t f dxdv    \\
=&\int_\Omega \frac{\delta}{\delta f}\mathcal{E}(f)\cdot\nabla_{x,v}\cdot(f(aa^{\ts}+\mathsf J)\nabla_{x,v}\frac{\delta}{\delta f}\mathcal{E}) dxdv \\
=&-\int_\Omega (\nabla_{x,v}\frac{\delta}{\delta f}\mathcal{E}, (aa^{\ts}+\mathsf J)\nabla_{x,v}\frac{\delta}{\delta f}\mathcal{E}(f))fdxdv\\
=&-\int_\Omega (\nabla_{x,v}\frac{\delta}{\delta f}\mathcal{E}, aa^{\ts}\nabla_{x,v}\frac{\delta}{\delta f}\mathcal{E}(f))fdxdv,
\end{split}
\end{equation*}
where we use the fact that 
\begin{equation*}
   (\nabla_{x,v}\frac{\delta}{\delta f}\mathcal{E},\mathsf J\nabla_{x,v}\frac{\delta}{\delta f}\mathcal{E}(f))=0. 
\end{equation*} \qed 
\end{proof}

\begin{lemma}[Technical Lemma]\label{lemma: main lemma-R}
Suppose Assumption \ref{main assumption} holds. Then
\beaa 
\pa_t\mathcal{DE}_{a,z}(f)\le -2\int_{\Omega\times\Omega} \mathfrak R(\delta\mathcal E,\delta \mathcal E) f(t,x,v)f(t,y,\tilde v) dxdvdyd\tilde v\le -2\lambda [ \mathcal {DE}_a(f)+ \mathcal {DE}_z(f)].
\eeaa 
\end{lemma}
We leave the proof of Lemma \ref{lemma: main lemma-R} in section \ref{section5}. 


\noindent \textbf{Proof of Theorem~\ref{main theorem}} 
According to Lemma \ref{lemma: main lemma-R}, we have
\beaa 
\pa_t \mathcal{DE}_{a,z}(f)
\le -2\lambda \mathcal{DE}_{a,z}(f).
\eeaa
Furthermore, using the fact that 
$$
-\frac{d}{dt}\mathcal{E}(f)=\mathcal{DE}_a(f)\le \mathcal{DE}_{a,z}(f),$$ 
we have 
\beaa
-[\mathcal{DE}_{a,z}(f)-\mathcal{DE}_{a,z}(f_{\infty})]&=&\int_t^{\infty} \frac{d}{ds}\mathcal{DE}_{a,z}(f(s,\cdot,\cdot))ds\\
\text{Step A:}&\le & -2\lambda \int_t^{\infty} \mathcal{DE}_{a,z}(f(s,\cdot,\cdot))ds\\
&\le & -2\lambda \int_t^{\infty} \mathcal{DE}_{a}(f(s,\cdot,\cdot))ds\\
&=&-2\lambda \int_t^{\infty}-\frac{d}{ds}\mathcal{E}(f(s,\cdot,\cdot))ds\\
\text{Step B:}&=&-2\lambda [\mathcal E(f)- \mathcal E(f_{\infty}) ].
\eeaa
{
From Step B, we have 
\begin{equation}\label{AB}
[\mathcal E(f)- \mathcal E(f_{\infty}) ]\le \frac{1}{2\lambda}[\mathcal{DE}_{a,z}(f)-\mathcal{DE}_{a,z}(f_{\infty})]. 
\end{equation}
From Step A, we have 
\beaa
-[\mathcal{DE}_{a,z}(f)-\mathcal{DE}_{a,z}(f_{\infty})]\le -2\lambda \int_t^{\infty} \mathcal{DE}_{a,z}(f(s,\cdot,\cdot))-\mathcal{DE}_{a,z}(f_{\infty})ds.
\eeaa 
Applying Gronwall's inequality, we have 
\beaa 
\mathcal{DE}_{a,z}(f)-\mathcal{DE}_{a,z}(f_{\infty})&\le&e^{-2\lambda t}[\mathcal{DE}_{a,z}(f_0)-\mathcal{DE}_{a,z}(f_{\infty})].
\eeaa 
Using the inequality \eqref{AB}, we have 
\begin{equation*}
 \mathcal E(f)- \mathcal E(f_{\infty}) \leq \frac{1}{2\lambda} \Big[\mathcal{DE}_{a,z}(f)-\mathcal{DE}_{a,z}(f_{\infty})\Big]\le \frac{1}{2\lambda}e^{-2\lambda t}\Big[\mathcal{DE}_{a,z}(f_0)-\mathcal{DE}_{a,z}(f_{\infty})\Big],
\end{equation*}
which finishes the proof.}
\qed 


\noindent \textbf{Proof of Corollary \ref{cor 2}} Recall that 
	\beaa
	\cE(f)&=&\int_\Omega f\log f dxdv+\frac{1}{2}\int_\Omega \|v\|^2 f dxdv\\
	&&+\frac{1}{2}\int_{\Omega\times\Omega} W(x,y)f(x,v)f(y,\tilde v)dx dv dy d\tilde v+\int_{\Omega} U(x)f(x,v)dxdv.
	\eeaa
Similarly, we denote $\cE(f_{\infty})$ as the free energy associated with the equilibrium $f_{\infty}$. We know that $f_{\infty}$ satisfies the equation, 
\[
f_{\infty}(x,v)=\frac{1}{Z}e^{-\frac{1}{2}\|v\|^2-\int_\Omega W(x,y)f_{\infty}(y,v)dydv-U(x)},
\]
which implies that 
\[
\log f_{\infty}=-\frac{1}{2}\|v\|^2-\int_\Omega W(x,y)f_{\infty}(y,v)dydv-U(x)-\log Z,
\]
where $Z$ is the normalization constant. We thus obtain 
\beaa
\cE(f_{\infty})=-\frac{1}{2}\int_{\Omega\times \Omega} W(x,y) f_{\infty}(x,v)f_{\infty}(y,\tilde v)dxdvdyd\tilde v-\log Z.
\eeaa
Following the above representation of $\cE(f_{\infty})$, and denoting $f=f(t,\cdot,\cdot)$, we derive 
\beaa
&&\cE(f)-\cE(f_{\infty})\\
&=&\int_\Omega f(t,x,v)\log f(t,x,v) dxdv+\frac{1}{2}\int_\Omega \|v\|^2 f(t,x,v) dxdv\\
&&+\frac{1}{2}\int_{\Omega\times\Omega} W(x,y)f(t,x,v)f(t,y,\tilde v)dx dv dy d\tilde v+\int_{\Omega}U(x)f(t,x,v)dxdv\\
&&+\frac{1}{2}\int_{\Omega\times \Omega} W(x,y) f_{\infty}(x,v)f_{\infty}(y,\tilde v)dxdvdyd\tilde v+\log Z
\eeaa 
\beaa 
&=& \int_\Omega f\log \frac{f}{f_{\infty}}dxdv+\int_\Omega f \log f_{\infty} dxdv+\frac{1}{2}\int_\Omega \|v\|^2 f dxdv+\int_{\Omega}U(x)fdxdv+\log Z\\
&&+\frac{1}{2}\int_{\Omega\times\Omega} W(x,y)f(t,x,v)f(t,y,\tilde v)dx dv dy d\tilde v+\frac{1}{2}\int_{\Omega\times\Omega} W(x,y) f_{\infty}(x,v)f_{\infty}(y,\tilde v)dxdvdyd\tilde v\\
&=& \int_{\Omega} f\log \frac{f}{f_{\infty}}dxdv-\int_{\Omega\times\Omega} f(t,x,\tilde v) W(x,y)f_{\infty}(y,v)dydvdxd\tilde v\\
&&+\frac{1}{2}\int_{\Omega\times\Omega} W(x,y)f(t,x,v)f(t,y,\tilde v)dx dv dy d\tilde v+\frac{1}{2}\int_{\Omega\times\Omega} W(x,y) f_{\infty}(x,v)f_{\infty}(y,\tilde v)dxdvdyd\tilde v\\
&=& \int_\Omega f\log \frac{f}{f_{\infty}}dxdv +\frac{1}{2}\int_{\Omega\times\Omega} W(x,y)(f(t,x,v)-f_{\infty}(x,v))(f(t,y,\tilde v)-f_{\infty}(y,\tilde v))dxdydvd\tilde v,\\
&\ge& \int_\Omega f\log \frac{f}{f_{\infty}}dxdv -\frac{1}{2}\int_{\Omega\times\Omega} |W(x,y)|\cdot|(f(t,x,v)-f_{\infty}(x,v))|\cdot|(f(t,y,\tilde v)-f_{\infty}(y,\tilde v))|dxdydvd\tilde v\\
&\ge &  \int f\log \frac{f}{f_{\infty}}dxdv-\frac{1}{2}C_{W}\|f-f_{\infty}\|_{L^1}^2,
\eeaa 
where the last inequality follows from our assumption in \eqref{assump bound on W}. Applying the Csisz\'ar-Kullback-Pinsker inequality, we have
\beaa
\cE(f)-\cE(f_{\infty})&\ge& \int_\Omega f\log\frac{f}{f_{\infty}}dxdv-\frac{1}{2}C_{W}\|f-f_{\infty}\|_{L^1}^2\\
&\ge&\frac{1}{2} \|f-f_{\infty}\|^2_{L^1}-\frac{1}{2}C_{W}\|f-f_{\infty}\|_{L^1}^2.
\eeaa
The proof is thus completed assuming that $C_{W}<1$. \qed

\section{Proofs of technical lemmas}\label{section5}
In this section, we provide all proofs of technical lemmas in the previous section. For simplicity of notations, we use the integration notation that $\int=\int_\Omega$. Besides, we denote $\Omega=\Omega_x\times\Omega_v$, $\Omega_x=\mathbb{T}^d$, $\Omega_v=\mathbb{R}^d$, and use the notation $\rho$ to represent the marginal density function of $f$ on the spatial domain, i.e., 
\bea
\rho(t,x)=\int_{\Omega_v}f(t,x,v)dv.
\eea

\subsection{Notations and Information Gamma operators}
In this subsection, we prepare some notations for later on computations in technical lemmas. 

Denote $\phi\in C^\infty(\Omega)$ as a smooth testing function. 
Denote the Kolmogorov operator $L:=L(f)$ for PDE \eqref{Kinetic FP} as 
\begin{equation*}
L\phi= v\nabla_x \phi-\nabla_v\phi\cdot \nabla_x (\int_{\Omega} W(x,y)f(t,y,\widetilde v)d\widetilde vdy+U(x))+\Delta_v \phi-v\cdot\nabla_v \phi.
\end{equation*} 
We also denote the $L^2$ adjoint operator with respect to the Lebesgue measure of $L$ as $L^*$.  In other words, \eqref{Kinetic FP} can be written as 
\begin{equation*}
\partial_tf=L^*f.     
\end{equation*}
Following Proposition \ref{prop: decomposition of FP}, we decompose the operator $L$ as 
\begin{equation*}
\begin{aligned}
L\phi=\widetilde L\phi -\la \gamma,\nabla \phi\ra,
\end{aligned}
\end{equation*}
where 
  \beaa 
\widetilde L \phi&=& \nabla\cdot(aa^{\ts}\nabla \phi)+\la aa^{\ts}
\nabla(-\frac{v^2}{2}-U(x)-W\oast\rho(x)), \nabla \phi\ra.
\eeaa
Denote a $z$-direction generator:
\begin{equation*}
\widetilde L_z \phi= \nabla\cdot(zz^{\ts}\nabla \phi)+\la zz^{\ts}
\nabla(-\frac{v^2}{2}-U(x)-W\oast\rho(x)), \nabla \phi\ra.    
\end{equation*}
We first define Gamma one bilinear forms $\Gamma_{1},\Gamma_{1}^z\colon C^{\infty}(\Omega)\times C^{\infty}(\Omega)\rightarrow C^{\infty}(\Omega) $ by 
\bea
\Gamma_{1}(\phi,\phi)=\la a^{\ts}\nabla \phi, a^{\ts}\nabla \phi\ra_{\hR^d},\quad \Gamma_{1}^z(\phi,\phi)=\la z^{\ts}\nabla \phi, z^{\ts}\nabla \phi\ra_{\hR^d}.
\eea
Next, we recall the following \emph{information gamma calculus} introduced in \cite{FengLi2021}, where $a$ and $z$ are chosen as constant matrices. 
\label{defn:tilde gamma 2 znew}
Define the following three bi-linear forms:
\begin{equation*}
 \widetilde\Gamma_{2}, \widetilde\Gamma_{2}^{z}, \Gamma_{\mathcal{I}_{a,z}}\colon C^{\infty}(\Omega)\times C^{\infty}(\Omega)\rightarrow C^{\infty}(\Omega).
\end{equation*}
We denote 
\beaa 
\widetilde\Gamma_{2}(\phi, \phi)=\frac{1}{2}\widetilde L\Gamma_{1}(\phi,\phi)-\Gamma_{1}(\widetilde L\phi, \phi),
\eeaa 
and 
\beaa
\widetilde\Gamma_2^{z}(\phi,\phi)=\frac{1}{2}\widetilde L\Gamma_1^{z}(\phi,\phi)-\Gamma_{1}^z(\widetilde L\phi,\phi).
  \eeaa 
 We denote the irreversible Gamma operator: 
\bea\label{irr gamma operator}
\Gamma_{\mathcal{I}_{a,z}}(\phi,\phi)&=& (\widetilde L\phi+\widetilde L_z\phi) \la \nabla \phi,\gamma\ra -\frac{1}{2}\la \nabla \big(\Gamma_1(\phi,\phi)+\Gamma_1^z(\phi,\phi)\big),\gamma\ra.
\eea
We first present the following proposition for $a$ and $z$ defined in \eqref{matrix a z}, which is a special case of \cite{FengLi2021}[Proposition 9]. See other (generalized) Bakry-Emery formulation in \cite{baudoin2016curvature} for degenerate operator $L$ and  \cite{bakry1985, arnold2000generalized} for non-degenerate operator $L$. 
\begin{proposition}\label{bochner formula}
For any $\phi(x,v)\in C^{\infty}(\Omega)$, we have
\bea 
\widetilde\Gamma_2(\phi,\phi)+\widetilde\Gamma_2^z(\phi,\phi)&=&\sum_{i=1}^d\Big[z_1^2|\pa^2_{x_ix_i}\phi|^2+(1+z_2^2)|\pa^2_{x_{i+d}x_{i+d}}\phi|^2+2z_1z_2|\pa^2_{x_ix_{i+d}}\phi|^2\Big]\nonumber\\
&&-\sum_{i=1}^d\sum_{\hat k=1}^{2d}a_i^{\ts}\nabla [aa^{\ts}\nabla(-\frac{v^2}{2}-U(x)-W\oast\rho(x))]_{\hat k} \partial_{x_{\hat k}}\phi a^{\ts}_i\nabla \phi \nonumber\\
&&-\sum_{k=1}^d\sum_{\hat k=1}^{2d}z_k^{\ts}\nabla [aa^{\ts}\nabla(-\frac{v^2}{2}-U(x)-W\oast\rho(x))]_{\hat k} \partial_{x_{\hat k}}\phi z^{\ts}_k\nabla \phi,\nonumber
\eea 
where we denote $(x_1,\cdots x_d,v_1,\cdots,v_d)=(x_1,\cdots,x_{2d})$, and $\nabla \phi=(\pa_{x_1}f,\cdots, \pa_{x_{2d}}\phi)$. Furthermore, we denote $a^{\ts}=(a^{\ts}_1,\cdots,a^{\ts}_d)^{\ts}\in\hR^{d\times 2d}$ and $z^{\ts}=(z^{\ts}_1,\cdots,z^{\ts}_d)^{\ts}\in\hR^{d\times 2d}$.
\end{proposition}
We next prove the following equivalent formulation for the irreversible Gamma operator. For simplicity of presentation, we use the following notation in the rest of the paper:
\bea
\delta\mathcal E=\frac{\delta}{\delta f}\mathcal E(f).
\eea 
\begin{lemma}\label{lemma gamma I}
For irreversible Gamma operator $\Gamma_{\mathcal I_{a,z}}=\Gamma_{\mathcal I_{a}}+\Gamma_{\mathcal I_{z}}$ defined in \eqref{irr gamma operator}, we have
	\beaa
	\int \Gamma_{\mathcal I_a}(\delta\mathcal E,\delta\mathcal E)fdxdv
&=&
-\int \la aa^{\ts}\nabla \delta\mathcal E, \nabla \gamma \nabla \delta\mathcal E\ra f dxdv,\\ 
\int \Gamma_{\mathcal I_z}(\delta\mathcal E,\delta\mathcal E)f dxdv
&=&
-\int \la zz^{\ts}\nabla \delta\mathcal E, \nabla \gamma \nabla \delta\mathcal E\ra f dxdv.
\eeaa
\end{lemma}
\begin{proof}
We lay out the proof for the first identity with matrix $a$. The second identity for matrix $z$ can be proved in a similar manner. 
Recall 
\begin{equation*}
    \Gamma_{\mathcal{I}_{a,z}}(\delta\mathcal E,\delta\mathcal E)= (\widetilde L\delta\mathcal E+\widetilde L_z\delta\mathcal E) \la \nabla \delta\mathcal E,\gamma\ra -\frac{1}{2}\la \nabla \big(\Gamma_1(\delta\mathcal E,\delta\mathcal E)+\Gamma_1^z(\delta\mathcal E,\delta\mathcal E)\big),\gamma\ra=\Gamma_{\mathcal{I}_{a}}(\delta\mathcal E,\delta\mathcal E)+\Gamma_{\mathcal{I}_{z}}(\delta\mathcal E,\delta\mathcal E),
    \end{equation*}
    and 
    \beaa 
\widetilde L \delta\mathcal E&=& \nabla\cdot(aa^{\ts}\nabla\delta\mathcal E)+\la aa^{\ts}
\nabla(-\frac{v^2}{2}-U(x)-W\oast\rho(x)), \nabla \delta\mathcal E\ra,\\
\widetilde L_z \delta\mathcal E&=& \nabla\cdot(zz^{\ts}\nabla\delta\mathcal E)+\la zz^{\ts}
\nabla(-\frac{v^2}{2}-U(x)-W\oast\rho(x)), \nabla \delta\mathcal E\ra.\eeaa
Then we have
	\beaa
&&\int \Gamma_{\mathcal I_{a}}(\delta\mathcal E,\delta\mathcal E)f dxdv\\
&=&\int \Big[(\widetilde L\delta\mathcal E)\la \nabla \delta\mathcal E,\gamma\ra-\frac{1}{2}\la \nabla(\Gamma_1(\delta\mathcal E,\delta\mathcal E)),\gamma\ra  \Big]fdxdv  \\
&=&\int \Big[\nabla\cdot ((aa^{\ts})\nabla \delta\mathcal E)\la \nabla \delta\mathcal E,\gamma\ra +\la \nabla \delta\mathcal E,\gamma\ra  \la \nabla \delta\mathcal E, (aa^{\ts})\nabla(-\frac{v^2}{2}-U(x)-W\oast\rho(x)) \ra \Big] f dxdv\\
&&+ \int \frac{1}{2}\nabla\cdot(f \gamma ) \la \nabla \delta\mathcal E,(aa^{\ts})\nabla \delta\mathcal E\ra dxdv.\eeaa
Using the fact $\nabla\cdot(f\gamma)=f\la \nabla\delta\mathcal E,\gamma \ra=f\la\nabla \delta\mathcal E,\gamma\ra $, we have 
\bea\label{le eq 1}
&&\int \Gamma_{\mathcal I_{a}}(\delta\mathcal E,\delta\mathcal E)f dxdv \\
&=&\int \Big[\nabla\cdot ((aa^{\ts})\nabla \delta\mathcal E)\la \nabla \delta\mathcal E,\gamma\ra +\la \nabla \delta\mathcal E,\gamma\ra  \la \nabla \delta\mathcal E, aa^{\ts}\nabla(-\frac{v^2}{2}-U(x)-W\oast\rho(x)) \ra \Big] f dxdv\nonumber\\
&&+\int \frac{1}{2}\la  \nabla \delta\mathcal E, \gamma \ra  \la \nabla \delta\mathcal E,aa^{\ts}\nabla \delta\mathcal E\ra fdxdv.\nonumber
\eea
Applying integration by parts for the first term, we have
\beaa
&&\int \nabla\cdot (aa^{\ts}\nabla \delta\mathcal E)\la \nabla \delta\mathcal E,\gamma\ra  f dxdv\\
&=&- \int\Big[ \la  aa^{\ts}\nabla \delta\mathcal E,\nabla \log f \ra \la \nabla \delta\mathcal E,\gamma\ra  +\la aa^{\ts}\nabla \delta\mathcal E,\nabla^2 \delta\mathcal E \gamma\ra \Big] f dxdv\\
&&-\int \la  aa^{\ts}\nabla \delta\mathcal E,\nabla \gamma  \nabla \delta\mathcal E\ra f dxdv.
\eeaa
Plugging the above equality in \eqref{le eq 1} and using the fact $\nabla\delta \mathcal E=\nabla\log f-\nabla( -\frac{v^2}{2}-U(x)-W\oast\rho(x))$, we have
\beaa
&&\int \Gamma_{\mathcal I_{a}}(\delta\mathcal E,\delta\mathcal E)f dx\\
&=& \int -\frac{1}{2}(\nabla \delta\mathcal E,\gamma) (\nabla \delta\mathcal E, aa^{\ts}\nabla \delta\mathcal E)fdxdv\\
&&- \int\Big[ \la  aa^{\ts}\nabla \delta\mathcal E,\nabla \gamma \nabla \delta\mathcal E\ra  +\la aa^{\ts}\nabla \delta\mathcal E,\nabla^2 \delta\mathcal E \gamma\ra \Big] f dxdv\\
&=& \int -\frac{1}{2}\la \nabla f,\gamma\ra  \la \nabla \delta\mathcal E, aa^{\ts}\nabla\delta\mathcal E\ra dxdv\\
&& +\int\frac{1}{2}\la \nabla( -\frac{v^2}{2}-U(x)-W\oast\rho(x)), \gamma\ra   \la \nabla \delta\mathcal E, aa^{\ts}\nabla \delta\mathcal E\ra  fdxdv\\
&&- \int\Big[ \la  aa^{\ts}\nabla \delta\mathcal E,\nabla \gamma \nabla \delta\mathcal E\ra  +\la aa^{\ts}\nabla \delta\mathcal E,\nabla^2 \delta\mathcal E \gamma\ra \Big] f dxdv. 
\eeaa
Applying integration by parts for the first term and using the identity $\nabla\cdot(\gamma)=0$, we have
\beaa
&&\int -\frac{1}{2}(\nabla f,\gamma) (\nabla \delta\mathcal E, aa^{\ts}\nabla \delta\mathcal E)dxdv\\
&=&\int \frac{1}{2} \nabla\cdot ( \gamma \la \nabla \delta\mathcal E, aa^{\ts}\nabla \delta\mathcal E\ra )f dxdv\\
&=&\int \frac{1}{2} \la \gamma,\la \nabla \delta\mathcal E,\nabla(aa^{\ts})\nabla\delta\mathcal E\ra\ra +\int\la aa^{\ts}\nabla \delta\mathcal E,\nabla^2 \delta\mathcal E \gamma\ra  f dxdv.
\eeaa
By summing over all above terms, we obtain 
\beaa 
\int \Gamma_{\mathcal I_{a}}(\delta\mathcal E,\delta\mathcal E)fdxdv
=-\int  \la \nabla \gamma\nabla \delta\mathcal E,aa^{\ts}\nabla \delta\mathcal E\ra fdxdv
\eeaa
In the above equality, we use the fact that $\la \nabla( -\frac{v^2}{2}-U(x)-W\oast\rho(x)), \gamma\ra=0$  and $\nabla(aa^{\ts})=0$ since $a$ is constant matrix. This completes the proof.
\qed 
\end{proof}

\subsection{Proof of Lemma \ref{lemma: main lemma-R}}
The proof of Lemma \ref{lemma: main lemma-R} is divided into the following lemmas.  
\begin{lemma}\label{lemma: a}
For $\mathcal{DE}_a(f)$ defined in \eqref{fisher}, we have the following equality:
	\beaa
	\pa_t\mathcal{DE}_a(f)&=&-2\int[\widetilde\Gamma_2(\delta\mathcal E,\delta\mathcal E)-\la aa^{\ts}\nabla \delta\mathcal E, \nabla \gamma \nabla \delta\mathcal E\ra ] fdxdv\\
	&&-2\int  \la \nabla^2_{xy}W(x,y) (\mathsf J +aa^{\ts})\nabla_{y,\tilde v}\delta\mathcal E(y,\tilde v), aa^{\ts} \nabla_{x,v}\delta \mathcal E(x,v) \ra f(x,v)f(y,\tilde v) dxdvdyd\tilde v.
	\eeaa 
\end{lemma}

\begin{proof}
For $\mathcal{DE}_a(f)$ defined in \eqref{fisher}, and the structure of matrix $a$, we have 
	\beaa
\mathcal{DE}_a(f)&=&\int\la \nabla \delta \mathcal E, aa^{\ts}\nabla \delta \mathcal E\ra fdxdv.
	\eeaa 
	We derive the dissipation of $\mathcal{DE}_a(f)$ as below, where we denote $\nabla=\nabla_{x,v}$, 
\beaa 
&&\pa_t \mathcal{DE}_a(f)\\
&=&2\int\la \nabla \pa_t\delta \mathcal E, aa^{\ts}\nabla \delta \mathcal E\ra fdxdv+\int\la \nabla \delta \mathcal E, aa^{\ts}\nabla \delta \mathcal E\ra\pa_t fdxdv\\
&=&-2\int\la (\pa_t \delta \mathcal E)(x)\nabla\cdot (f aa^{\ts}\nabla \delta \mathcal E)dxdv+\int\la \nabla \delta \mathcal E, aa^{\ts}\nabla \delta \mathcal E\ra\pa_t fdxdv\\
&=&-2\int   (\delta^2 \mathcal E(x,v,y,\tilde v)\pa_t f(y,\tilde v)\nabla\cdot (f aa^{\ts}\nabla \delta \mathcal E)(x,v)dxdv+\int\la \nabla \delta \mathcal E, aa^{\ts}\nabla \delta \mathcal E\ra\pa_t fdxdv\\
&=& -2\int   (W(x,y)\pa_t f(y,\tilde v))\nabla\cdot (f aa^{\ts}\nabla \delta \mathcal E)(x,v)dxdvdyd\tilde v- 2\int   (\frac{1}{f}\pa_t f)\nabla\cdot (f aa^{\ts}\nabla \delta \mathcal E) dxdv\\
&&+\int\la \nabla \delta \mathcal E, aa^{\ts}\nabla \delta \mathcal E\ra\pa_t fdxdv.
\eeaa 
Hence
\beaa 
&&\pa_t \mathcal{DE}_a(f)\\
&=&-2\int W(x,y)[\nabla_{y,\tilde v}\cdot (f aa^{\ts}\nabla_{y,\tilde v}\delta\mathcal E(y,\tilde v))+\nabla_{y,\tilde v}\cdot (f\gamma  )]\nabla_{x,v}\cdot (faa^{\ts} \nabla_{x,v}\delta \mathcal E(x,v) )dxdvdyd\tilde y\\
&&-2\int \frac{1}{f}[\nabla_{x,v}\cdot (f aa^{\ts}\nabla_{x, v}\delta\mathcal E)+\nabla_{x, v}\cdot (f\gamma  )]\nabla_{x,v}\cdot (faa^{\ts} \nabla_{x,v}\delta \mathcal E )dxdv\\
&&+\int\la \nabla \delta \mathcal E, aa^{\ts}\nabla \delta \mathcal E\ra[\nabla\cdot (f aa^{\ts}\nabla_{x, v}\delta\mathcal E)+\nabla_{x, v}\cdot (f\gamma  )]dxdv\\ 
&=&-2\int W(x,y)\nabla_{y,\tilde v}\cdot (f aa^{\ts}\nabla_{y,\tilde v}\delta\mathcal E(y,\tilde v))\nabla_{x,v}\cdot (faa^{\ts} \nabla_{x,v}\delta \mathcal E(x,v) )dxdvdyd\tilde v\cdots \mathcal J_{11}\\
&&-2\int W(x,y)\nabla_{y,\tilde v}\cdot (f\gamma  )\nabla_{x,v}\cdot (faa^{\ts} \nabla_{x,v}\delta \mathcal E(x,v) )dxdvdyd\tilde v\cdots \mathcal J_{21}\\
&&-2\int \frac{1}{f}\nabla_{x,v}\cdot (f aa^{\ts}\nabla_{x, v}\delta\mathcal E)\nabla_{x,v}\cdot (faa^{\ts} \nabla_{x,v}\delta \mathcal E )dxdv\cdots \mathcal J_{12}\\
&&-2\int \frac{1}{f}\nabla_{x, v}\cdot (f\gamma  )\nabla_{x,v}\cdot (faa^{\ts} \nabla_{x,v}\delta \mathcal E )dxdv\cdots \mathcal J_{22}
\eeaa 
\beaa 
&&+\int\la \nabla \delta \mathcal E, aa^{\ts}\nabla \delta \mathcal E\ra\nabla\cdot (f aa^{\ts}\nabla_{x, v}\delta\mathcal E)dxdv\cdots \mathcal J_{13}\\
&&+\int\la \nabla \delta \mathcal E, aa^{\ts}\nabla \delta \mathcal E\ra \nabla_{x, v}\cdot (f\gamma  )dxdv\cdots \mathcal J_{23}\\
&=&\mathcal J_{11}+ \mathcal J_{12}+ \mathcal J_{13}+\mathcal J_{21}+\mathcal J_{22}+\mathcal J_{23}.
\eeaa
We first have
\beaa 
\mathcal J_{11}=-2\int \la \nabla_{xy}^2 W(x,y) aa^{\ts}\nabla_{y,\tilde v}\delta\mathcal E(y,\tilde v), aa^{\ts} \nabla_{x,v}\delta \mathcal E(x,v) \ra f(x,v)f(y,\widetilde v )dxdvdyd\tilde v.
\eeaa 
Next, using the identity $\nabla_{x,v}\cdot (f\gamma)=\nabla\cdot (f\mathsf J\nabla\delta\mathcal E)$, we have 
\beaa 
\mathcal J_{21}&=&-2\int W(x,y)\nabla_{y,\tilde v}\cdot (f\gamma  )\nabla_{x,v}\cdot (faa^{\ts} \nabla_{x,v}\delta \mathcal E )dxdvdyd\tilde y\\
&=& -2\int  \la \nabla^2_{xy}W(x,y)\mathsf J \nabla_{y,\tilde v}\delta\mathcal E(y,\tilde v), aa^{\ts} \nabla_{x,v}\delta \mathcal E(x,v) \ra f(x,v)f(y,\tilde v) dxdvdyd\tilde v.
\eeaa  
Furthermore, applying the identity $
\nabla_{x,v}\cdot (f\gamma)=f \la \nabla\delta\mathcal E,\gamma\ra$, and  we have 
\beaa 
&&\mathcal J_{22}+\mathcal J_{23}\\
&=&-2\int \nabla_{x,v}\cdot(f\gamma) \frac{ \nabla_{x,v}\cdot (faa^{\ts} \nabla_{x,v}\delta \mathcal E ) }{f}fdxdv+\int\la \nabla \delta \mathcal E, aa^{\ts}\nabla \delta \mathcal E\ra \nabla_{x, v}\cdot (f\gamma  )  dxdv\\
&=& -2 \int \nabla_{x,v}\cdot(f\gamma) \Big[\la  \nabla_{x,v} \log f,  aa^{\ts} \nabla_{x,v}\delta \mathcal E \ra +  \nabla_{x,v}\cdot (aa^{\ts} \nabla_{x,v}\delta \mathcal E )\Big]dxdv\\
&&+\int\la \nabla \delta \mathcal E, aa^{\ts}\nabla \delta \mathcal E\ra \nabla_{x,v}\cdot(f\gamma) dxdv\\ 
&=&-2 \int \nabla_{x,v}\cdot(f\gamma) \Big[\la  \nabla_{x,v} \delta \mathcal E,  aa^{\ts} \nabla_{x,v}\delta \mathcal E \ra + \widetilde L\delta \mathcal E\Big]dxdv+\int\la \nabla \delta \mathcal E, aa^{\ts}\nabla \delta \mathcal E\ra \nabla_{x,v}\cdot(f\gamma) dxdv\\
&=&-2\int \la \nabla\delta \mathcal E,\gamma\ra \tilde L\delta \mathcal E  f dxdv-\int\la \nabla \delta \mathcal E, aa^{\ts}\nabla \delta \mathcal E\ra \nabla_{x,v}\cdot(f\gamma)dxdv\\
&=& -2\Big[\int \la \nabla\delta \mathcal E,\gamma\ra \tilde L\delta \mathcal E f dxdv-\frac{1}{2}\int
\la \nabla \Gamma_1(\delta\mathcal E,\delta\mathcal E),\gamma\ra  f dxdv \Big]\\
&=& -2\int\widetilde \Gamma_{\mathcal I_a}(\delta\mathcal E,\delta\mathcal E)fdxdv.
\eeaa 
Similarly, we have 
\beaa 
\mathcal J_{12}+\mathcal J_{13}&=&-2\int \frac{1}{f}\nabla\cdot (f aa^{\ts}\nabla_{x, v}\delta\mathcal E)\nabla_{x,v}\cdot (faa^{\ts} \nabla_{x,v}\delta \mathcal E )dxdv\\
&&+\int\la \nabla \delta \mathcal E, aa^{\ts}\nabla \delta \mathcal E\ra\nabla\cdot (f aa^{\ts}\nabla_{x, v}\delta\mathcal E)dxdv\\
&=&-2\int \Big[\la  \nabla_{x,v} \delta \mathcal E,  aa^{\ts} \nabla_{x,v}\delta \mathcal E \ra + \widetilde L\delta \mathcal E\Big]\widetilde L^*fdxdv
+\int\la \nabla \delta \mathcal E, aa^{\ts}\nabla \delta \mathcal E\ra\widetilde L^* f dxdv\\
&=&-2 \int\Big[\frac{1}{2} \widetilde L \Gamma_1(\delta\mathcal E,\delta\mathcal E)-\Gamma_{1}(\widetilde L\delta\mathcal E,\delta\mathcal E ) \Big]fdxdv \\
&=& -2 \int \widetilde \Gamma_2 (\delta\mathcal E,\delta\mathcal E)fdxdv.
\eeaa 
Applying Lemma \ref{lemma gamma I} and combining the above terms, we finish the proof.
	\qed 
\end{proof}

\begin{lemma}\label{lemma: z}
For $\mathcal{DE}_z(f)$ defined in \eqref{fisher z}, we have the following equality
	\beaa
	\pa_t\mathcal{DE}_z(f)&=&-2\int[\widetilde\Gamma_2^z(\delta\mathcal E,\delta\mathcal E)-\la zz^{\ts}\nabla \delta\mathcal E, \nabla \gamma \nabla \delta\mathcal E\ra]  fdxdv\\
	&&-2\int  \la \nabla^2_{xy}W(x,y) (\mathsf J +aa^{\ts})\nabla_{y,\tilde v}\delta\mathcal E(y,\tilde v), zz^{\ts} \nabla_{x,v}\delta \mathcal E(x,v) \ra f(x,v)f(y,\tilde v) dxdvdyd\tilde v.
	\eeaa 
\end{lemma}
\begin{proof}
For $\mathcal{DE}_z(f)$ defined in \eqref{fisher}, and the structure of matrix $a$, we have 
	\beaa
\mathcal{DE}_z(f)&=&\int\la \nabla \delta \mathcal E, zz^{\ts}\nabla \delta \mathcal E\ra fdxdv.
	\eeaa 
	We derive the dissipation of $\mathcal{DE}_z(f)$ as below:
\beaa 
&&\pa_t \mathcal{DE}_z(f)\\
&=&2\int\la \nabla \pa_t\delta \mathcal E, zz^{\ts}\nabla \delta \mathcal E\ra fdxdv+\int\la \nabla \delta \mathcal E, zz^{\ts}\nabla \delta \mathcal E\ra\pa_t fdxdv\\
&=&-2\int\la (\pa_t \delta \mathcal E)(x)\nabla\cdot (f zz^{\ts}\nabla \delta \mathcal E)dxdv+\int\la \nabla \delta \mathcal E, zz^{\ts}\nabla \delta \mathcal E\ra\pa_t fdxdv\\
&=&-2\int   (\delta^2 \mathcal E(x,v,y,\tilde v)\pa_t f(y,\tilde v)\nabla\cdot (f zz^{\ts}\nabla \delta \mathcal E)(x,v)dxdv+\int\la \nabla \delta \mathcal E, zz^{\ts}\nabla \delta \mathcal E\ra\pa_t fdxdv\\
&=& -2\int   (W(x,y)\pa_t f(y,\tilde v))\nabla\cdot (f zz^{\ts}\nabla \delta \mathcal E)(x,v)dxdvdyd\tilde v- 2\int   (\frac{1}{f}\pa_t f)\nabla\cdot (f zz^{\ts}\nabla \delta \mathcal E) dxdv\\
&&+\int\la \nabla \delta \mathcal E, zz^{\ts}\nabla \delta \mathcal E\ra\pa_t fdxdv.
\eeaa 
Plugging in the equation for $\pa_tf$, we have
\beaa 
&&\pa_t \mathcal{DE}_z(f)\\
&=&-2\int W(x,y)[\nabla_{y,\tilde v}\cdot (f aa^{\ts}\nabla_{y,\tilde v}\delta\mathcal E(y,\tilde v))+\nabla_{y,\tilde v}\cdot (f\gamma  )]\nabla_{x,v}\cdot (fzz^{\ts} \nabla_{x,v}\delta \mathcal E(x,v) )dxdvdyd\tilde y\\
&&-2\int \frac{1}{f}[\nabla_{x,v}\cdot (f aa^{\ts}\nabla_{x, v}\delta\mathcal E)+\nabla_{x, v}\cdot (f\gamma  )]\nabla_{x,v}\cdot (fzz^{\ts} \nabla_{x,v}\delta \mathcal E )dxdv\\
&&+\int\la \nabla \delta \mathcal E, zz^{\ts}\nabla_{x,v} \delta \mathcal E\ra[\nabla\cdot (f aa^{\ts}\nabla_{x, v}\delta\mathcal E)+\nabla_{x, v}\cdot (f\gamma  )]dxdv
\eeaa 
\beaa 
&=&-2\int W(x,y)\nabla_{y,\tilde v}\cdot (f aa^{\ts}\nabla_{y,\tilde v}\delta\mathcal E(y,\tilde v))\nabla_{x,v}\cdot (fzz^{\ts} \nabla_{x,v}\delta \mathcal E(x,v) )dxdvdyd\tilde v\cdots \mathcal J^z_{11}\\
&&-2\int W(x,y)\nabla_{y,\tilde v}\cdot (f\gamma  )\nabla_{x,v}\cdot (fzz^{\ts} \nabla_{x,v}\delta \mathcal E(x,v) )dxdvdyd\tilde v\cdots \mathcal J^z_{21}\\
&&-2\int \frac{1}{f}\nabla_{x,v}\cdot (f aa^{\ts}\nabla_{x, v}\delta\mathcal E)\nabla_{x,v}\cdot (fzz^{\ts} \nabla_{x,v}\delta \mathcal E )dxdv\cdots \mathcal J^z_{12}\\
&&-2\int \frac{1}{f}\nabla_{x, v}\cdot (f\gamma  )\nabla_{x,v}\cdot (fzz^{\ts} \nabla_{x,v}\delta \mathcal E )dxdv\cdots \mathcal J^z_{22}\\
&&+\int\la \nabla \delta \mathcal E, zz^{\ts}\nabla \delta \mathcal E\ra\nabla_{x,v}\cdot (f aa^{\ts}\nabla_{x, v}\delta\mathcal E)dxdv\cdots \mathcal J^z_{13}\\
&&+\int\la \nabla \delta \mathcal E, zz^{\ts}\nabla \delta \mathcal E\ra \nabla_{x, v}\cdot (f\gamma  )dxdv\cdots \mathcal J^z_{23}\\
&=&\mathcal J^z_{11}+ \mathcal J^z_{12}+ \mathcal J^z_{13}+\mathcal J^z_{21}+\mathcal J^z_{22}+\mathcal J^z_{23}.
\eeaa
We first have
\beaa 
\mathcal J^z_{11}&=&-2\int \la \nabla_{xy}^2 W(x,y) aa^{\ts}\nabla_{y,\tilde v}\delta\mathcal E(y,\tilde v), zz^{\ts} \nabla_{x,v}\delta \mathcal E(x,v) \ra f(x,v)f(y,\widetilde v )dxdvdyd\tilde v.
\eeaa 
Next, using the identity $\nabla_{x,v}\cdot (f\gamma)=\nabla_{x,v}\cdot (f\mathsf J\nabla_{x,v}\delta\mathcal E)$, we have 
\beaa 
\mathcal J^z_{21}&=&-2\int W(x,y)\nabla_{y,\tilde v}\cdot (f\gamma  )\nabla_{x,v}\cdot (fzz^{\ts} \nabla_{x,v}\delta \mathcal E )dxdvdyd\tilde y\\
&=& -2\int  \la \nabla^2_{xy}W(x,y)\mathsf J \nabla_{y,\tilde v}\delta\mathcal E(y,\tilde v), zz^{\ts} \nabla_{x,v}\delta \mathcal E(x,v) \ra f(x,v)f(y,\tilde v) dxdvdyd\tilde v.
\eeaa  
Furthermore, applying the identity $
\nabla_{x,v}\cdot (f\gamma)=f \la \nabla_{x,v}\delta\mathcal E,\gamma\ra$, we have 
\beaa
&&\mathcal J_{22}+\mathcal J_{23}\\
&=&-2\int \nabla_{x,v}\cdot(f\gamma) \frac{ \nabla_{x,v}\cdot (fzz^{\ts} \nabla_{x,v}\delta \mathcal E ) }{f}fdxdv+\int\la \nabla \delta \mathcal E, zz^{\ts}\nabla \delta \mathcal E\ra \nabla_{x, v}\cdot (f\gamma  )  dxdv\\
&=& -2 \int \nabla_{x,v}\cdot(f\gamma) \Big[\la  \nabla_{x,v} \log f,  zz^{\ts} \nabla_{x,v}\delta \mathcal E \ra +  \nabla_{x,v}\cdot (zz^{\ts} \nabla_{x,v}\delta \mathcal E )\Big]dxdv\\
&&+\int\la \nabla \delta \mathcal E, zz^{\ts}\nabla \delta \mathcal E\ra \nabla_{x,v}\cdot(f\gamma) dxdv\\ 
&=&-2 \int \nabla_{x,v}\cdot(f\gamma) \Big[\la  \nabla_{x,v} \delta \mathcal E,  zz^{\ts} \nabla_{x,v}\delta \mathcal E \ra + \widetilde L_z\delta \mathcal E\Big]dxdv+\int\la \nabla \delta \mathcal E, zz^{\ts}\nabla \delta \mathcal E\ra \nabla_{x,v}\cdot(f\gamma) dxdv\\
&=&-2\int \la \nabla\delta \mathcal E,\gamma\ra \tilde L_z\delta \mathcal E  f dxdv-\int\la \nabla \delta \mathcal E, zz^{\ts}\nabla \delta \mathcal E\ra \nabla_{x,v}\cdot(f\gamma)dxdv\\
&=& -2\Big[\int \la \nabla\delta \mathcal E,\gamma\ra \tilde L_z\delta \mathcal E f dxdv-\frac{1}{2}\int
\la \nabla \Gamma_1^z(\delta\mathcal E,\delta\mathcal E),\gamma\ra  f dxdv \Big]\\
&=& -2\int\widetilde \Gamma_{\mathcal I_z}(\delta\mathcal E,\delta\mathcal E)fdxdv.
\eeaa 
Similarly, we have 
\beaa 
\mathcal J_{12}^z+\mathcal J_{13}^z&=&-2\int \frac{1}{f}\nabla\cdot (f aa^{\ts}\nabla_{x, v}\delta\mathcal E)\nabla_{x,v}\cdot (fzz^{\ts} \nabla_{x,v}\delta \mathcal E )dxdv\\
&&+\int\la \nabla \delta \mathcal E, zz^{\ts}\nabla \delta \mathcal E\ra\nabla\cdot (f aa^{\ts}\nabla_{x, v}\delta\mathcal E)dxdv\\
&=&-2\int \Big[\la  \nabla_{x,v} \delta \mathcal E,  zz^{\ts} \nabla_{x,v}\delta \mathcal E \ra + \widetilde L_z\delta \mathcal E\Big]\widetilde L^*fdxdv
+\int\la \nabla \delta \mathcal E, zz^{\ts}\nabla \delta \mathcal E\ra\widetilde L^* f dxdv\\
&=&-2 \int\Big[\frac{1}{2} \widetilde L \Gamma_1^z(\delta\mathcal E,\delta\mathcal E)-\Gamma_{1}(\widetilde L_z\delta\mathcal E,\delta\mathcal E ) \Big]fdxdv \\
&=&-2 \int\Big[\frac{1}{2} \widetilde L \Gamma_1^z(\delta\mathcal E,\delta\mathcal E)-\Gamma_1^z(\widetilde L\delta\mathcal E,\delta\mathcal E)+\Gamma_1^z(\widetilde L\delta\mathcal E,\delta\mathcal E)- \Gamma_{1}(\widetilde L_z\delta\mathcal E,\delta\mathcal E ) \Big]fdxdv \\
&=& -2 \int \widetilde \Gamma_2 (\delta\mathcal E,\delta\mathcal E)fdxdv,
\eeaa 
where the last equality follows from the following observation 
\bea\label{tri linear form zero}
\int [\Gamma_1^z(\widetilde L\delta\mathcal E,\delta\mathcal E)- \Gamma_{1}(\widetilde L_z\delta\mathcal E,\delta\mathcal E ) ]fdxdv=0,
\eea 
for constant matrices $a$ and $z$. We first observe that 
\begin{equation*}
\begin{split}
&    \int \Gamma_1^z( \widetilde L\delta\mathcal E, \delta\mathcal E) fdxdv\\
=& \int\Gamma_1^z\Big(\la\nabla( -\frac{v^2}{2}-U(x)-W\oast\rho(x)), aa^{\ts}\nabla \delta\mathcal E\ra+\nabla\cdot(aa^{\ts}\nabla\delta\mathcal E), \delta\mathcal E\Big)fdxdv\\
=&\int\Gamma_1^z\Big(\Gamma_1( -\frac{v^2}{2}-U(x)-W\oast\rho(x), \delta\mathcal E), \delta\mathcal E\Big)f +\Gamma_1^z\Big(\nabla\cdot(aa^{\ts}\nabla\delta\mathcal E), \delta\mathcal E\Big)fdxdv\\
=&\int\Gamma_1^z\Big(\Gamma_1( -\frac{v^2}{2}-U(x)-W\oast\rho(x), \delta\mathcal E), \delta\mathcal E\Big)f dxdv-\int \nabla\cdot(aa^{\ts}\nabla\delta\mathcal E)\nabla\cdot(f zz^{\ts}\nabla\delta\mathcal E)dxdv\\
=&\int\Gamma_1^z\Big(\Gamma_1( -\frac{v^2}{2}-U(x)-W\oast\rho(x), \delta\mathcal E), \delta\mathcal E\Big)f dxdv-\int \nabla\cdot(aa^{\ts}\nabla\delta\mathcal E)\nabla\cdot(f zz^{\ts}\nabla\delta\mathcal E)dxdv\\
=&\int\Gamma_1^z\Big(\Gamma_1( -\frac{v^2}{2}-U(x)-W\oast\rho(x), \delta\mathcal E), \delta\mathcal E\Big)f dxdv-\int \nabla\cdot(aa^{\ts}\nabla\delta\mathcal E)\nabla\cdot( zz^{\ts}\nabla\delta\mathcal E)fdxdv\\
&-\int \nabla\cdot(aa^{\ts}\nabla\delta\mathcal E)\la \nabla\log f,zz^{\ts}\nabla\delta\mathcal E\ra fdxdv.
\end{split}
\end{equation*}
Similarly, we have 
\begin{equation*}
\begin{split}
&    \int \Gamma_1( \widetilde L_z\delta\mathcal E, \delta\mathcal E) fdxdv\\
=&\int\Gamma_1\Big(\Gamma_1^z( -\frac{v^2}{2}-U(x)-W\oast\rho(x), \delta\mathcal E), \delta\mathcal E\Big)f dxdv-\int \nabla\cdot(aa^{\ts}\nabla\delta\mathcal E)\nabla\cdot( zz^{\ts}\nabla\delta\mathcal E)fdxdv\\
&-\int \nabla\cdot(zz^{\ts}\nabla\delta\mathcal E)\la \nabla\log f,aa^{\ts}\nabla\delta\mathcal E\ra fdxdv.
\end{split}
\end{equation*}
Furthermore, by direct expansion of the gradient, we have 
\beaa 
&&\int \nabla\cdot(aa^{\ts}\nabla\delta\mathcal E)\la \nabla\log f,zz^{\ts}\nabla\delta\mathcal E\ra fdxdv\\
&=&-\int \la a^{\ts}\nabla\delta\mathcal E, a^{\ts}\nabla \la z^{\ts}\nabla\log f,z^{\ts}\nabla\delta\mathcal E\ra\ra  fdxdv\\
&=& -\int \sum_{i=1}^d\sum_{j=1}^d\sum_{\hat i,\hat j,\tilde i,\tilde j=1}^{2d}\Big(a^{\ts}_{i\hat i}\pa_{x_{\hat i}}\delta\mathcal E a^{\ts}_{i\tilde i}\pa_{x_{\tilde i}}[z^{\ts}_{j\hat j}\pa_{x_{\hat j}}\log fz^{\ts}_{j\tilde j}\pa_{x_{\tilde j}}\delta\mathcal E ]\Big)  f dxdv\\
&=&-\int \sum_{i=1}^d\sum_{j=1}^d\sum_{\hat i,\hat j,\tilde i,\tilde j=1}^{2d}\Big( a^{\ts}_{i\hat i}a^{\ts}_{i\tilde i}z^{\ts}_{j\hat j}z^{\ts}_{j\tilde j}\pa_{x_{\hat i}}\delta\mathcal E \pa^2_{x_{\hat j}x_{\tilde i}}\log f\pa_{x_{\tilde j}}\delta\mathcal E \Big)  f dxdv\\
&&-\int \sum_{i=1}^d\sum_{j=1}^d\sum_{\hat i,\hat j,\tilde i,\tilde j=1}^{2d}\Big( a^{\ts}_{i\hat i}a^{\ts}_{i\tilde i}z^{\ts}_{j\hat j}z^{\ts}_{j\tilde j}\pa_{x_{\hat i}}\delta\mathcal E \pa_{x_{\hat j}}\log f\pa^2_{x_{\tilde j}x_{\tilde i}}\delta\mathcal E \Big)  f dxdv\\
&=&\int \nabla\cdot(zz^{\ts}\nabla\delta\mathcal E)\la \nabla\log f,aa^{\ts}\nabla\delta\mathcal E\ra fdxdv,
\eeaa 
where we denote $(x_1,\cdots,x_d,v_1,\cdots,v_d)=(x_1,\cdots,x_{2d})$ and use the fact that $a$ and $z$ are constant matrices. Similarly, we have 
\beaa 
&&\int\Gamma_1^z\Big(\Gamma_1( -\frac{v^2}{2}-U(x)-W\oast\rho(x), \delta\mathcal E), \delta\mathcal E\Big)f dxdv\\
&=&\int\Gamma_1\Big(\Gamma_1^z( -\frac{v^2}{2}-U(x)-W\oast\rho(x), \delta\mathcal E), \delta\mathcal E\Big)f dxdv,
\eeaa 
which proves equation \eqref{tri linear form zero}. Applying Lemma \ref{lemma gamma I} and combining the above terms, we complete the proof.  \qed 
\end{proof}

We are now ready to prove Lemma \ref{lemma: main lemma-R}. Recall the mean-field Langevin dynamics
\begin{equation}\label{SDE}
\left\{\begin{aligned}
       dx_t=&v_t dt \\
   dv_t=&(- v_t-\nabla_x\widetilde V(x_t, f))dt+\sqrt{2}dB_t,
\end{aligned}\right.
\end{equation}
where $f$ is the solution of PDE \eqref{Kinetic FP}, and we denote
\begin{equation*}
\widetilde V(x,f)=\int_{\Omega} W(x,y)f(t,y,\tilde v)dyd\tilde v+U(x). 
\end{equation*}

\begin{proof}[Proof of Lemma \ref{lemma: main lemma-R}]
We first present the explicit formulation of matrix $\mathfrak R$.
Following from Lemma \ref{lemma: a} and Lemma \ref{lemma: z}, we have 
	\beaa
	&&\pa_t\mathcal{DE}_{a,z}(f)\\
	&=&-2\int[\widetilde\Gamma_2(\delta\mathcal E,\delta\mathcal E)+\widetilde\Gamma_2^z(\delta\mathcal E,\delta\mathcal E)-\la (aa^{\ts}+zz^{\ts})\nabla \delta\mathcal E, \nabla \gamma \nabla \delta\mathcal E\ra]  fdxdv\\
	&&-2\int \nabla^2_{xy}W(x,y) (\la\mathsf J +aa^{\ts})\nabla_{y,\tilde v}\delta\mathcal E(y,\tilde v), (aa^{\ts} +zz^{\ts})\nabla_{x,v}\delta \mathcal E(x,v) \ra f(x,v)f(y,\tilde v) dxdvdyd\tilde v.
	\eeaa 
Following from the definition of $\widetilde\Gamma_2$ and $\widetilde\Gamma_2^z$ from Section 4 and Proposition \ref{bochner formula} we have 
\beaa 
\widetilde\Gamma_2(\delta\mathcal E,\delta\mathcal E)+\widetilde\Gamma_2^z(\delta\mathcal E,\delta\mathcal E)
&=&\sum_{i=1}^d\Big[z_1^2|\pa^2_{x_ix_i}\delta\mathcal{E}|^2+(1+z_2^2)|\pa^2_{x_{i+d}x_{i+d}}\delta\mathcal{E}|^2+2z_1z_2|\pa^2_{x_ix_{i+d}}\delta\mathcal{E}|^2\Big]\nonumber\\
&&-\sum_{i=1}^d\sum_{\hat k=1}^{2d}a_i^{\ts}\nabla [aa^{\ts}\nabla(-\frac{v^2}{2}-U(x)-W\oast\rho(x))]_{\hat k} \partial_{x_{\hat k}}\delta\mathcal{E}a^{\ts}_i\nabla \delta\mathcal{E} \nonumber\\
&&-\sum_{k=1}^d\sum_{\hat k=1}^{2d}z_k^{\ts}\nabla [aa^{\ts}\nabla(-\frac{v^2}{2}-U(x)-W\oast\rho(x))]_{\hat k} \partial_{x_{\hat k}}\delta\mathcal{E}z^{\ts}_k\nabla \delta\mathcal{E},\nonumber\\
&=&\|\mathfrak{Hess}\delta\mathcal{E}\|_{\text{F}}^2+\mathfrak R_{a,z}\mathfrak (\delta\mathcal E,\delta\mathcal{E}),
\eeaa 
where we define
\beaa 
\|\mathfrak{Hess}\delta\mathcal{E}\|_{\text{F}}^2&=& \sum_{i=1}^d\Big[z_1^2|\pa^2_{x_ix_i}\delta\mathcal{E}|^2+(1+z_2^2)|\pa^2_{x_{i+d}x_{i+d}}\delta\mathcal{E}|^2+2z_1z_2|\pa^2_{x_ix_{i+d}}\delta\mathcal{E}|^2\Big]\\
\mathfrak R_{a,z}(\delta\mathcal E,\delta\mathcal{E})&=&-\sum_{i=1}^d\sum_{\hat k=1}^{2d}a_i^{\ts}\nabla [aa^{\ts}\nabla(-\frac{v^2}{2}-U(x)-W\oast\rho(x))]_{\hat k} \partial_{x_{\hat k}}\delta\mathcal{E}a^{\ts}_i\nabla \delta\mathcal{E} \nonumber\\
&&-\sum_{k=1}^d\sum_{\hat k=1}^{2d}z_k^{\ts}\nabla [aa^{\ts}\nabla(-\frac{v^2}{2}-U(x)-W\oast\rho(x))]_{\hat k} \partial_{x_{\hat k}}\delta\mathcal{E}z^{\ts}_k\nabla \delta\mathcal{E}.
\eeaa 

According to the above computation, we have
\beaa
&& \int_\Omega[ (\widetilde\Gamma_2+ \widetilde \Gamma_2^{z} )(\delta\mathcal E,\delta\mathcal E)
-  \la (aa^{\ts}+zz^{\ts})\nabla \delta\mathcal E, \nabla \gamma \nabla \delta\mathcal E\ra ]fdxdv\\
&&+\int_{\Omega\times\Omega}  \la \nabla^2_{xy}W(x,y) (\mathsf J +aa^{\ts})\nabla_{y,\tilde v}\delta\mathcal E(y,\tilde v), (aa^{\ts} +zz^{\ts})\nabla_{x,v}\delta \mathcal E(x,v) \ra f(x,v)f(y,\tilde v) dxdvdyd\tilde v\\
&\ge& \mathfrak R_1+\mathfrak R_2,
\eeaa
where we denote 
\beaa 
\mathfrak R_1&=&
\int_{\Omega\times\Omega} \la  \nabla^2_{xy}W(x,y)(\mathsf J +aa^{\ts})\nabla_{y,\tilde v}\delta\mathcal E(y,\tilde v), (aa^{\ts} +zz^{\ts})\nabla_{x,v}\delta \mathcal E(x,v) \ra f(x,v)f(y,\tilde v) dxdvdyd\tilde v.\\
\mathfrak R_2&=&\int_{\Omega}[ \mathfrak R_{a,z}\mathfrak (\delta\mathcal E,\delta\mathcal{E})-\la (aa^{\ts}+zz^{\ts})\nabla \delta\mathcal E, \nabla \gamma \nabla \delta\mathcal E\ra] f(x,v)dxdv.
\eeaa 
If Assumption \ref{main assumption} holds, we have 
\beaa 
&&\int_{\Omega}[ (\widetilde\Gamma_2+ \widetilde \Gamma_2^{z} )(\delta\mathcal{E},\delta\mathcal{E}) 
-  \la (aa^{\ts}+zz^{\ts})\nabla \delta\mathcal{E}, \nabla \gamma \nabla \delta\mathcal{E}\ra 
]fdxdv\\
&&+\int_{\Omega\times\Omega}  \la\nabla^2_{xy}W(x,y)(\mathsf J +aa^{\ts})\nabla_{y,\tilde v}\delta\mathcal E(y,\tilde v), (aa^{\ts} +zz^{\ts})\nabla_{x,v}\delta \mathcal E(x,v) \ra f(x,v)f(y,\tilde v) dxdvdyd\tilde v\\
 &\ge &\mathfrak R_1+\mathfrak R_2= \int_{\Omega\times\Omega}\mathfrak{R}(\nabla\delta\mathcal{E},\nabla\delta\mathcal{E})f(t,x,v)f(t,y,\tilde v)dxdvdyd\tilde v\\
 &\ge&  \lambda\int_\Omega  (\Gamma_1(\delta\mathcal{E},\delta\mathcal{E})+\Gamma_1^z(\delta\mathcal{E},\delta\mathcal{E}))fdxdv
 =\lambda [\mathcal{DE}_a(f)+\mathcal{DE}_z(f)]. 
\eeaa
Thus, we only need to show that $\mathfrak R$ is indeed the matrix defined in Definition \ref{definition mean field info matrix}. With some abuse of notations, we denote $\mathfrak R=\mathfrak R_1+\mathfrak R_2$. In the following, we derive the explicit formulation of $\mathfrak R$ as defined in Definition \ref{definition mean field info matrix} and Definition \ref{definition mean field info matrix-2}.

\noindent\textbf{ Case 1:} For d=1, the $\mathfrak R$ has the following two parts. for constant matrices $a=(0\quad 1)^{\ts}$ and $z=(z_1\quad z_2)^{\ts}$,
\beaa
 \mathfrak R_1&=&
 \int  \la \nabla^2_{xy}W(x,y)(\mathsf J +aa^{\ts})\nabla_{y,\tilde v}\delta\mathcal E(y,\tilde v), (aa^{\ts} +zz^{\ts})\nabla_{x,v}\delta \mathcal E(x,v) \ra f(x,v)f(y,\tilde v) dxdvdyd\tilde v.\\
&=&\int (\nabla\delta\mathcal E(y,\tilde v))^{\ts} \mathsf{sym}\Big((aa^{\ts}+zz^{\ts})^{\ts}\nabla^2_{xy}W(x,y)(aa^{\ts}+\mathsf J) \Big)(\nabla\delta\mathcal E(x,v))f(x,v)f(y,\tilde v)dxdvdyd\tilde v.
\eeaa 
To be precise, we denote $\int=\int_{\Omega\times\Omega}$ in $\mathfrak R_1$.
Plugging in the matrices $a$, $z$, and $\mathsf J$, we have the following symmetrization of the matrix,  
\beaa 
&&\mathsf{sym}\Big((aa^{\ts}+zz^{\ts})^{\ts}\nabla^2_{xy}W(x,y)(aa^{\ts}+\mathsf J) \Big)\\
&=&\mathsf{sym}\Big((aa^{\ts}+zz^{\ts})\nabla^2_{xy}W(x,y)(aa^{\ts}+\mathsf J) \Big)\\
&=& \mathsf{sym}\Big(\begin{pmatrix}
	z_1^2 & z_1z_2 \\
	z_1z_2& (1+z_2^2)
\end{pmatrix}_{2\times 2}\nabla^2_{xy}W(x,y)
\begin{pmatrix}
	0&-1\\
	1&1
\end{pmatrix}_{2\times 2}  \Big)\\
&=&\mathsf{sym}\Big(\begin{pmatrix}
	z_1^2 & z_1z_2 \\
	z_1z_2& (1+z_2^2)
\end{pmatrix}_{2\times 2}\begin{pmatrix}
0&-	\nabla^2_{xy}W(x,y)\\
	0& 0
\end{pmatrix}_{2d\times2d}  \Big)\\
&=&\mathsf{sym}\begin{pmatrix}
	0&-z_1^2\nabla^2_{xy}W(x,y)\\
0&-z_1z_2\nabla^2_{xy}W(x,y)
\end{pmatrix} \\
&=&\begin{pmatrix}
0&-\frac{z_1^2}{2}\nabla^2_{xy}W(x,y)\\
	-\frac{z_1^2}{2}\nabla^2_{xy}W(x,y)&-z_1z_2\nabla^2_{xy}W(x,y)
\end{pmatrix}.
\eeaa 
This implies that 
\beaa 
\mathfrak R_1=\int (\nabla\delta\mathcal E(y,\tilde v))^{\ts} \begin{pmatrix}
0&-\frac{z_1^2}{2}\nabla^2_{xy}W(x,y)\\
	-\frac{z_1^2}{2}\nabla^2_{xy}W(x,y)&-z_1z_2\nabla^2_{xy}W(x,y)
\end{pmatrix}(\nabla\delta\mathcal E(x,v)) f(x,v)f(y,\tilde v)dxdvdyd\tilde v.
\eeaa 
As for the second term, we have 
\beaa 
\mathfrak R_2&=&\int_{\Omega}[ \mathfrak R_{a,z}\mathfrak (\delta\mathcal E,\delta\mathcal{E})-\la (aa^{\ts}+zz^{\ts})\nabla \delta\mathcal E, \nabla \gamma \nabla \delta\mathcal E\ra] f(x,v)dxdv\\
&=&\int_{\Omega}(\nabla\delta\mathcal{E})^{\ts}\mathfrak R_2 \nabla\delta\mathcal{E} f(x,v)dxdv.
\eeaa 
By direct computation and matrix symmetrization, the matrix $\mathfrak R_2$ has the following representation, 
\beaa
\mathfrak{R}_2&=&\begin{pmatrix}
0&0\\
0& -\frac{\partial^2 (-\frac{v^2}{2}-\widetilde V(x,f) ) }{\partial v^2}
\end{pmatrix}+\frac{1}{2}\Big[  \begin{pmatrix}
0\\
-z^{\ts}_1\nabla( \frac{\partial  (-\frac{v^2}{2}-\widetilde V(x,f) )}{\partial v})
\end{pmatrix} z^{\ts}+z \begin{pmatrix}
0&
-z^{\ts}_1\nabla( \frac{\partial  (-\frac{v^2}{2}-\widetilde V(x,f) )}{\partial v})
\end{pmatrix} \Big]\\
&&-\frac{1}{2}[(\nabla \gamma)^{\ts}aa^{\ts}+aa^{\ts}\nabla\gamma]
-\frac{1}{2}[(\nabla \gamma)^{\ts}zz^{\ts}+zz^{\ts}\nabla\gamma],
\eeaa
with $(\nabla\gamma)_{ij}=\nabla_i\gamma_j$, such that 
 \[ 
 \nabla\gamma=\begin{pmatrix}
 	0& \int \nabla^2_{xx} W(x,y)f(t,y,v)dydv+\nabla_{xx}U(x)\\
 	-1& 0
 \end{pmatrix}=\begin{pmatrix}
 	0&\nabla_{xx}^2\widetilde V(x,f)\\
 	-1&0
 \end{pmatrix} , 
 \]
 and 
\[ zz^{\ts}=\begin{pmatrix}
 	z_1^2&z_1z_2\\
 	z_1z_2&z_2^2
 \end{pmatrix},\quad aa^{\ts}=\begin{pmatrix}
 	0&0\\
 	0&1
 \end{pmatrix}.
 \]
 By direct computations, we have  
 \beaa
\mathfrak{R}_2&=&\begin{pmatrix}
0&0\\
0& 1
\end{pmatrix}+\frac{1}{2}\Big[  \begin{pmatrix}
0\\
z_2
\end{pmatrix} (z_1,z_2)+\begin{pmatrix}
	z_1\\
	z_2
\end{pmatrix} \begin{pmatrix}
0&
z_2
\end{pmatrix} \Big]-\frac{1}{2}[(\nabla \gamma)^{\ts}aa^{\ts}+aa^{\ts}\nabla\gamma]
\\
&&-\frac{1}{2}[(\nabla \gamma)^{\ts}zz^{\ts}+zz^{\ts}\nabla\gamma]
 \\
&=&\begin{pmatrix}
	0&0\\
	0&1
\end{pmatrix}+\begin{pmatrix}
	0&\frac{1}{2}z_1z_2\\
	\frac{1}{2}z_1z_2&z_2^2
\end{pmatrix}-\begin{pmatrix}
	0& -\frac{1}{2}\\
	-\frac{1}{2}& 0
\end{pmatrix}-\begin{pmatrix}
	-z_1z_2& \frac{1}{2}[\nabla^2\widetilde Vz_1^2-z_2^2 ]\\
	\frac{1}{2}[\nabla^2\widetilde Vz_1^2-z_2^2 ]& \nabla^2 \widetilde V z_1z_2
\end{pmatrix}\\
&=&\begin{pmatrix}
	z_1z_2& \frac{1}{2}[z_1z_2+z_2^2-\nabla^2\widetilde V z_1^2+1 ]\\
	\frac{1}{2}[z_1z_2+z_2^2-\nabla^2\widetilde V z_1^2+1 ]& 1+z_2^2-\nabla^2\widetilde V z_1z_2
\end{pmatrix}.
\eeaa
For notation simplicity, we denote $\mathsf U(x,v)=\nabla \delta\mathcal E(x,v)$ and $\mathsf U(y,\tilde v)=\nabla \delta\mathcal E(y,\tilde v)$. Recall that $\widetilde V(x,f)=\int_{\Omega} W(x,y)f(t,y,\tilde v)dyd\tilde v+U(x)$, we have
\beaa
&&\mathfrak R_1+\mathfrak R_2\\ 
&=&\int_{\Omega\times\Omega}\mathsf U^{\ts}(y,\tilde v)\begin{pmatrix}
0&-\frac{z_1^2}{2}\nabla^2_{xy}W(x,y)\\
	-\frac{z_1^2}{2}\nabla^2_{xy}W(x,y)&-z_1z_2\nabla^2_{xy}W(x,y)
\end{pmatrix}\mathsf U(x,v) f(x,v)f(y,\tilde v)dxdvdyd\tilde v \\
&&+\int_{\Omega}\mathsf U^{\ts}(x,v)\begin{pmatrix}
	z_1z_2& \frac{1}{2}[z_1z_2+z_2^2-\nabla^2_{xx}\widetilde V z_1^2+1 ]\\
	\frac{1}{2}[z_1z_2+z_2^2-\nabla^2_{xx}\widetilde V z_1^2+1 ]& 1+z_2^2-\nabla^2_{xx}\widetilde V z_1z_2
\end{pmatrix}  \mathsf U(x,v) f(x,v) dxdv\\ 
&=&\int_{\Omega\times\Omega}\mathsf U^{\ts}(y,\tilde v)\begin{pmatrix}
0&-\frac{z_1^2}{2}\nabla^2_{xy}W(x,y)\\
	-\frac{z_1^2}{2}\nabla^2_{xy}W(x,y)&-z_1z_2\nabla^2_{xy}W(x,y)
\end{pmatrix}\mathsf U(x,v) f(x,v)f(y,\tilde v)dxdvdyd\tilde v \\
&&+\int_{\Omega}\mathsf U^{\ts}(x,v)\begin{pmatrix}
	z_1z_2& \frac{1}{2}[z_1z_2+z_2^2-\nabla^2_{xx}U z_1^2+1 ]\\
	\frac{1}{2}[z_1z_2+z_2^2-\nabla^2_{xx}U z_1^2+1 ]& 1+z_2^2-\nabla^2_{xx} U z_1z_2
\end{pmatrix}  \mathsf U(x,v) f(x,v) dxdv\\
&&+\int_{\Omega\times \Omega}\mathsf U^{\ts}(x,v)\begin{pmatrix}
	0& -\frac{z_1^2}{2}\nabla^2_{xx}W(x,y)  \\
	-\frac{z_1^2}{2}\nabla^2_{xx}W(x,y)& -z_1z_2\nabla^2_{xx} W(x,y)
\end{pmatrix}  \mathsf U(x,v)f(y,\tilde v) f(x,v)dyd\tilde v dxdv\\
&=& \mathcal T_1 +\mathcal T_2+\mathcal T_3.
\eeaa 
Using the fact that $W(x,y)=W(y,x)$, and $\nabla_{xx}^2W(x,y)=\nabla^2_{yy}W(y,x)$, we have
\beaa 
&&\mathcal T_1+\mathcal T_3\\
&=&\frac{1}{2} \int_{\Omega\times \Omega}\mathsf U^{\ts}(x,v)\begin{pmatrix}
	0& -\frac{z_1^2}{2}\nabla^2_{xx}W(x,y)  \\
	-\frac{z_1^2}{2}\nabla^2_{xx}W(x,y)& -z_1z_2\nabla^2_{xx} W(x,y)
\end{pmatrix}  \mathsf U(x,v)f(y,\tilde v) f(x,v)dyd\tilde v dxdv\\
&&+\frac{1}{2}\int_{\Omega\times \Omega}\mathsf U^{\ts}(y,\tilde v)\begin{pmatrix}
	0& -\frac{z_1^2}{2}\nabla^2_{yy}W(y,x)  \\
	-\frac{z_1^2}{2}\nabla^2_{yy}W(y,x)& -z_1z_2\nabla^2_{yy} W(y,x)
\end{pmatrix}  \mathsf U(y,\tilde v)f(y,\tilde v) f(x,v)dyd\tilde v dxdv\\
&&+\int_{\Omega\times\Omega}\mathsf U^{\ts}(y,\tilde v)\begin{pmatrix}
0&-\frac{z_1^2}{2}\nabla^2_{xy}W(x,y)\\
	-\frac{z_1^2}{2}\nabla^2_{xy}W(x,y)&-z_1z_2\nabla^2_{xy}W(x,y)
\end{pmatrix}\mathsf U(x,v) f(x,v)f(y,\tilde v)dxdvdyd\tilde v\\
&=&\int_{\Omega\times\Omega}\Big[ \frac{1}{2}\mathsf U^{\ts}(x,v)\mathsf A_1(x,y) \mathsf U(x,v)+\frac{1}{2}\mathsf U^{\ts}(y,\tilde v)\mathsf A_1(y,x) \mathsf U(y,\tilde v)\\
&&\qq\qq+\mathsf U^{\ts}(y,\tilde v)\mathsf B(x,y) \mathsf U(x,v)\Big]f(x,v)f(y,\tilde v)dxdvdyd\tilde v
\eeaa 
\beaa 
&=&\frac{1}{2}\int_{\Omega\times\Omega} \begin{pmatrix}
	\mathsf U^{\ts}(x,v) & \mathsf U^{\ts}(y,\tilde v)
\end{pmatrix}\begin{pmatrix}
	\mathsf A_1(x,y)&	\mathsf B(x,y)\\
		\mathsf B(x,y)&	\mathsf A_1(y,x)
\end{pmatrix}\begin{pmatrix}
	\mathsf U(x,v)\\
	\mathsf U(y,\tilde v)
\end{pmatrix} f(x,v)f(y,\tilde v)dxdvdyd\tilde v,
\eeaa 
where we denote 
\beaa
\mathsf A_1(x,y)&=& \begin{pmatrix}
	0& -\frac{z_1^2}{2}\nabla^2_{xx}W(x,y)  \\
	-\frac{z_1^2}{2}\nabla^2_{xx}W(x,y)& -z_1z_2\nabla^2_{xx} W(x,y)
\end{pmatrix},\\
\mathsf A_1(y,x)&=& \begin{pmatrix}
	0& -\frac{z_1^2}{2}\nabla^2_{yy}W(y,x)  \\
	-\frac{z_1^2}{2}\nabla^2_{yy}W(y,x)& -z_1z_2\nabla^2_{yy} W(y,x)
\end{pmatrix}, \\
\mathsf B(x,y)&=& \begin{pmatrix}
0&-\frac{z_1^2}{2}\nabla^2_{xy}W(x,y)\\
	-\frac{z_1^2}{2}\nabla^2_{xy}W(x,y)&-z_1z_2\nabla^2_{xy}W(x,y)
\end{pmatrix}.
\eeaa
Combining $\mathcal T_2$ with the above term, we have 
\beaa
&&\mathcal T_1+\mathcal T_2+\mathcal T_3\\
&=&\frac{1}{2}\int_{\Omega\times\Omega} \begin{pmatrix}
	\mathsf U^{\ts}(x,v) & \mathsf U^{\ts}(y,\tilde v)
\end{pmatrix}\begin{pmatrix}
	\mathsf A_1(x,y)+2\mathsf A_2(x,y)&	\mathsf B(x,y)\\
		\mathsf B(x,y)&	\mathsf A_1(y,x)
\end{pmatrix}\begin{pmatrix}
	\mathsf U(x,v)\\
	\mathsf U(y,\tilde v)
\end{pmatrix}\\
&&\qq\qq f(x,v)f(y,\tilde v)dxdvdyd\tilde v\\
&=&\frac{1}{2}\int_{\Omega\times\Omega} \begin{pmatrix}
	\mathsf U^{\ts}(x,v) & \mathsf U^{\ts}(y,\tilde v)
\end{pmatrix}\begin{pmatrix}
	\mathsf A_1(x,y)+\mathsf A_2(x,y)&	\mathsf B(x,y)\\
		\mathsf B(x,y)&	\mathsf A_1(y,x)+\mathsf A_2(y,x)
\end{pmatrix}\begin{pmatrix}
	\mathsf U(x,v)\\
	\mathsf U(y,\tilde v)
\end{pmatrix}\\
&&\qq\qq f(x,v)f(y,\tilde v)dxdvdyd\tilde v,
\eeaa
where we denote 
\bea\label{A2 matrix} 
\mathsf A_2(x,y)&=&\begin{pmatrix}
	z_1z_2& \frac{1}{2}[z_1z_2+z_2^2-\nabla^2_{xx}U(x) z_1^2+1 ]\\
	\frac{1}{2}[z_1z_2+z_2^2-\nabla^2_{xx}U(x) z_1^2+1 ]& 1+z_2^2-\nabla^2_{xx} U(x) z_1z_2
\end{pmatrix}\nonumber \\
\mathsf A_2(y,x)&=&\begin{pmatrix}
	z_1z_2& \frac{1}{2}[z_1z_2+z_2^2-\nabla^2_{yy}U(y) z_1^2+1 ]\\
	\frac{1}{2}[z_1z_2+z_2^2-\nabla^2_{yy}U(y) z_1^2+1 ]& 1+z_2^2-\nabla^2_{yy} U(y) z_1z_2
\end{pmatrix}.
\eea
This produced the matrix tensor for Definition \ref{definition mean field info matrix-2}.
Combining the above matrix terms, we have
\beaa 
\mathfrak R=\frac{1}{2}\begin{pmatrix}
\mathsf A(x,y)&\mathsf B(x,y)\\
\mathsf B(x,y)&\mathsf A(y,x)
\end{pmatrix},
\eeaa 
where we denote
\beaa
\mathsf A(x,y)&=& \mathsf A_1(x,y)+\mathsf A_2(x,y)\\
&=&\begin{pmatrix}
	0& -\frac{z_1^2}{2}\nabla^2_{xx}W(x,y)  \\
	-\frac{z_1^2}{2}\nabla^2_{xx}W(x,y)& -z_1z_2\nabla^2_{xx} W(x,y)
\end{pmatrix}\\
&&+\begin{pmatrix}
	z_1z_2& \frac{1}{2}[(1+z_1z_2+z_2^2)-z_1^2\nabla^2_{xx}U(x) ]\\
	\frac{1}{2}[(1+z_1z_2+z_2^2)-z_1^2\nabla^2_{xx}U(x)]& (1+z_2^2)-z_1z_2\nabla^2_{xx} U(x) 
	\end{pmatrix}\\
	&=& \begin{pmatrix}
	z_1z_2& \frac{1}{2}[(1+z_1z_2+z_2^2)-z_1^2\nabla^2_{xx}V(x,y) ]\\
	\frac{1}{2}[(1+z_1z_2+z_2^2)-z_1^2\nabla^2_{xx}V(x,y)]& (1+z_2^2)-z_1z_2\nabla^2_{xx} V(x,y) 
	\end{pmatrix},
	\eeaa 
	and 
	\beaa 
\mathsf A(y,x)	&=& \begin{pmatrix}
	z_1z_2& \frac{1}{2}[(1+z_1z_2+z_2^2)-z_1^2\nabla^2_{yy}V(y,x) ]\\
	\frac{1}{2}[(1+z_1z_2+z_2^2)-z_1^2\nabla^2_{yy}V(y,x)]& (1+z_2^2)-z_1z_2\nabla^2_{yy} V(y,x) 
	\end{pmatrix},
	\eeaa
with $V(x,y)=W(x,y)+U(x)$, and	
\beaa	
	\mathsf B(x,y)&=& \begin{pmatrix}
0&-\frac{z_1^2}{2}\nabla^2_{xy}W(x,y)\\
	-\frac{z_1^2}{2}\nabla^2_{xy}W(x,y)&-z_1z_2\nabla^2_{xy}W(x,y)
\end{pmatrix}.
\eeaa
This produces the matrix tensor for Definition \ref{definition mean field info matrix}.

\noindent\textbf{Case 2: $d\ge 2$.}
We first demonstrate the derivation for $d=2$. By direction computations, we have 
\beaa
a=\begin{pmatrix}
	0&0&1&0\\
	0&0&0&1\\
\end{pmatrix}^{\ts}, \quad z=\begin{pmatrix}
	z_1&0&z_2&0\\
	0&z_1&0&z_2\\
\end{pmatrix}^{\ts},
\eeaa 
and 
\beaa 
&&\mathsf{sym}\Big((aa^{\ts}+zz^{\ts})^{\ts}\nabla^2_{xy}W(x,y)(aa^{\ts}+\mathsf J) \Big)\\
&=&\mathsf{sym}\Big((aa^{\ts}+zz^{\ts})\nabla^2_{xy}W(x,y)(aa^{\ts}+\mathsf J) \Big)\\
&=& \mathsf{sym}\Big(\begin{pmatrix}
	z_1^2\mathsf I_2 & z_1z_2 \mathsf I_2\\
	z_1z_2\mathsf I_2& (1+z_2^2)\mathsf I_2
\end{pmatrix}_{4\times 4}\begin{pmatrix} \nabla^2_{xy}W(x,y)&0\\
0&0\end{pmatrix}
\begin{pmatrix}
	0&-\mathsf I_2\\
	\mathsf I_2&\mathsf I_2
\end{pmatrix}_{2d\times 2d}  \Big)\\
&=&\begin{pmatrix}
0&-\frac{z_1^2}{2}\nabla^2_{xy}W(x,y)\\
	-\frac{z_1^2}{2}\nabla^2_{xy}W(x,y)&-z_1z_2\nabla^2_{xy}W(x,y)
\end{pmatrix}.
\eeaa 
Similar to $d=1$, by routine computation, we have
\beaa
\mathfrak R_2
&=& \begin{pmatrix}
	0&0&0&0&\\
	0&0&0&0&\\
	0&0&1&0&\\
	0&0&0&1&\\
\end{pmatrix}+\begin{pmatrix}
	0&0&\frac{1}{2} z_1z_2&0&\\
	0&0&0&\frac{1}{2}z_1z_2\\
	\frac{1}{2}z_1z_2&0&z_2^2&0\\
	0&\frac{1}{2}z_1z_2&0&z_2^2\\
\end{pmatrix}-\frac{1}{2}\begin{pmatrix}
	0&0&-1&0\\
	0&0&0&-1\\
	-1&0&0&0\\
	0&-1&0&0\\
\end{pmatrix}\\
&&-\begin{pmatrix}
	-z_1z_2&0&\frac{1}{2}(z_1^2\nabla_{x_1x_1}\widetilde V-z_2^2)&\frac{1}{2} z_1^2\nabla_{x_1x_2}\widetilde V\\
	0&-z_1z_2&\frac{1}{2} z_1^2\nabla_{x_1x_2}\widetilde V&\frac{1}{2}(z_1^2\nabla_{x_2x_2}\widetilde V-z_2^2)\\
	\frac{1}{2}(z_1^2\nabla_{x_1x_1}\widetilde V-z_2^2)&\frac{1}{2} z_1^2\nabla_{x_1x_2}\widetilde V& z_1z_2\nabla_{x_1x_1}\widetilde V&z_1z_2\nabla_{x_1x_2}\widetilde V\\
	\frac{1}{2} z_1^2\nabla_{x_1x_2}\widetilde V&\frac{1}{2}(z_1^2\nabla_{x_1x_1}\widetilde V-z_2^2)&z_1z_2\nabla_{x_1x_2}\widetilde V&z_1z_2\nabla_{x_2x_2}\widetilde V
\end{pmatrix},
\eeaa
where $\widetilde V(x,f)=\int_{\Omega} W(x,y)f(t,y,v)dvdy+U(x)$. We then combine the two matrices following the proof for $d=1$. Similarly, for any $d\ge 2$, we have 
\beaa 
\mathfrak R(z,x,y)=\frac{1}{2}\begin{pmatrix}
\mathsf A(x,y)&\mathsf B(x,y)\\
\mathsf B(x,y)&\mathsf A(y,x)
\end{pmatrix}\in\hR^{4d\times 4d},
\eeaa 
where
\beaa
\mathsf A(x,y)&=&\begin{pmatrix}
	0& -\frac{z_1^2}{2}\nabla^2_{xx}W(x,y)  \\
	-\frac{z_1^2}{2}\nabla^2_{xx}W(x,y)& -z_1z_2\nabla^2_{xx} W(x,y)
\end{pmatrix}, \\
&&+\begin{pmatrix}
	z_1z_2\mathsf I_d& \frac{1}{2}[(1+z_1z_2+z_2^2)\mathsf I_d-z_1^2\nabla^2_{xx}U(x) ]\\
	\frac{1}{2}[(1+z_1z_2+z_2^2)\mathsf I_d-z_1^2\nabla^2_{xx}U(x)]& (1+z_2^2)\mathsf I_d-z_1z_2\nabla^2_{xx} U(x) 
	\end{pmatrix}\\ 
&=&\begin{pmatrix}
	z_1z_2\mathsf I_d& \frac{1}{2}[(1+z_1z_2+z_2^2)\mathsf I_d-z_1^2\nabla^2_{xx}V(x,y)]\\
	\frac{1}{2}[(1+z_1z_2+z_2^2)\mathsf I_d-z_1^2\nabla^2_{xx}V(x,y)]& (1+z_2^2)\mathsf I_d-z_1z_2\nabla^2_{xx} V(x,y)
	\end{pmatrix},
	\eeaa
for $V(x,y)=W(x,y)+U(x)$, and	
\beaa	
	\mathsf B(x,y)&=& \begin{pmatrix}
0&-\frac{z_1^2}{2}\nabla^2_{xy}W(x,y)\\
	-\frac{z_1^2}{2}\nabla^2_{xy}W(x,y)&-z_1z_2\nabla^2_{xy}W(x,y)
\end{pmatrix}.
\eeaa
Separating the matrix $\mathsf A(x,y)$ into $\mathsf A_1(x,y)$ and $\mathsf A_2(x,y)$, we derive the matrix defined in Definition \ref{definition mean field info matrix-2}.
\qed 
\end{proof}

%
%
%

\end{document}